\documentclass{amsart}

\usepackage{amsmath,amsfonts,amssymb}
\usepackage[colorlinks=true]{hyperref}
\hypersetup{urlcolor=blue, citecolor=blue}
\usepackage{enumerate}
\usepackage{graphicx}
\usepackage{stmaryrd}

\newtheorem{thm}{Theorem}[section]
\newtheorem{lem}[thm]{Lemma}
\newtheorem{prop}[thm]{Proposition}
\newtheorem{coro}[thm]{Corollary}
\newtheorem{Def}[thm]{Definition}

\newtheorem{rem}[thm]{Remark}

\def\R{\mathbb{R}}

\def\be{\begin{equation}}
\def\ee{\end{equation}}
\def\g{{\bf g}}

\def\n{{\boldsymbol n}}
\def\p{\partial}
\def\grad{\boldsymbol{\nabla}}
\def\div{\grad\cdot}

\def\O{\Omega}
\def\G{\Gamma}

\def\sig{\sigma}

\def\a{\alpha}
\def\eps{\varepsilon}
\def\dive{\operatorname{div}}
\def\x{{\boldsymbol x}}
\def\d{{\rm d}}
\def\ov#1{\overline{#1}}
\def\un#1{\underline{#1}}
\def\wh#1{\widehat{#1}}
\def\wt#1{\widetilde{#1}}

\def\bga{\boldsymbol{\gamma}}
\def\balpha{\boldsymbol{\alpha}}
\def\bpi{{\boldsymbol{\pi}}}
\def\bxi{{\boldsymbol{\xi}}}

\def\bphi{{\boldsymbol{\phi}}}

\def\bUpsilon{\boldsymbol{\Upsilon}}
\def\bPsi{\boldsymbol{\Psi}}

\def\bF{{\boldsymbol  F}}

\def\bW{{\boldsymbol  W}}

\def\bh{{\boldsymbol  h}}

\def\bs{{\boldsymbol  s}}
\def\bp{{\boldsymbol  p}}
\def\bv{{\boldsymbol v}}
\def\bz{{\boldsymbol z}}

\def\by{{\boldsymbol y}}
\def\0{{\bf 0}}
\def\1{{\bf 1}}

\def\bbG{{\mathbb G}}

\def\bbI{{\mathbb I}}
\def\bbJ{{\mathbb J}}
\def\bbK{{\mathbb K}}

\def\bbS{{\mathbb S}}

\def\bAa{\boldsymbol{\mathcal A}}
\def\bXx{\boldsymbol{\mathcal X}}

\def\Aa{\mathcal{A}}

\def\Cc{\mathcal{C}}

\def\Ee{\mathcal{E}}
\def\Ff{\mathcal{F}}

\def\Hh{\mathcal{H}}

\def\Mm{\mathcal{M}}

\def\Oo{\mathcal{O}}

\def\Uu{\mathcal{U}}

\def\ICI{\begin{center} \color{red} \rule{\textwidth} {1pt} \\ICI \\\rule{\textwidth} {1pt} \end{center} } 

\begin{document}

\title[Multiphase flows in porous media: a variational approach]{Incompressible immiscible multiphase flows in porous media: a variational approach}
\author{Cl\'ement Canc\`es}
\address{
Cl\'ement Canc\`es (\href{mailto:clement.cances@inria.fr}{\tt clement.cances@inria.fr}).
Team RAPSODI, Inria Lille -- Nord Europe, 40 av. Halley, F-59650 Villeneuve d'Ascq, France.
} 
\author{Thomas O. Gallou\"et}
\address{
Thomas O. Gallou\"et (\href{mailto:thomas.gallouet@ulg.ac.be}{\tt thomas.gallouet@ulg.ac.be}).
D\'epartement de math\'ematiques, Universit\'e de Li\`ege, All\'ee de la d\'ecouverte 12, B-4000 Li\`ege, Belgique.
}
\author{L\'eonard Monsaingeon}
\address{
L\'eonard Monsaingeon (\href{mailto:leonard.monsaingeon@univ-lorraine.fr}{\tt leonard.monsaingeon@univ-lorraine.fr}).
Institut \'Elie Cartan de Lorraine, Universit\'e de Lorraine, Site de Nancy B.P. 70239, F-54506 Vandoeuvre-l\`ess-Nancy Cedex
}

\begin{abstract}
We describe the competitive motion of $(N+1)$ incompressible 
immiscible phases within a porous medium as the gradient flow 
of a singular energy in the space of non-negative measures with 
prescribed masses, endowed with some tensorial Wasserstein distance. 
We show the convergence of the approximation obtained by a 
minimization scheme {\em \`a la}  [R. Jordan, D. Kinderlehrer \& 
F. Otto, SIAM J. Math. Anal, 29(1):1--17, 1998]. This allow to obtain 
a new existence result for a physically well-established system of 
PDEs consisting in the Darcy-Muskat law for each phase, $N$ capillary 
pressure relations, and a constraint on the volume occupied by the fluid. 
Our study does not require the introduction of any global or 
complementary pressure. 
\end{abstract}

\maketitle

\noindent
{\small {\bf Keywords.}
Multiphase porous media flows, Wasserstein gradient flows, 
constrained parabolic system, minimizing movement scheme
\vspace{5pt}

\noindent
{\bf AMS subjects classification. }
35K65, 35A15, 49K20, 76S05
}

\tableofcontents

\newpage

\section{Introduction}\label{sec:intro}

\subsection{Equations for multiphase flows in porous media}\label{ssec:equation}
We consider a {convex} open bounded set $\O\subset \R^d$ representing a porous medium. 
$N+1$ incompressible and immiscible phases, labeled by subscripts $i\in \{0,\dots, N\}$ are supposed to flow within the pores. 
Let us present now some classical equations that describe the motion of such a mixture. The physical justification 
of these equations can be found for instance in~\cite[Chapter 5]{BB90}.
We denote by $s_i: \O\times(0,T)=:Q \to [0,1]$ the content of the phase $i$, i.e., the volume ratio of the phase $i$ compared to all the phases and the solid matrix, 
and by $\bv_i$ the filtration speed of the phase $i$.
Then the conservation of the volume of each phase writes
\be\label{eq:cons-mass}
\p_t s_i + \div\left(s_i \bv_i\right) = 0 \quad \text{in\;} Q, \quad \forall i \in \{0,\dots, N\}, 
\ee
where $T>0$ is an arbitrary finite time horizon.
The filtration speed of each phase is assumed to be given by Darcy's law
\be\label{eq:Darcy}
\bv_i = - \frac{1}{\mu_i} \bbK \left( \grad p_i - \rho_i \g\right) \quad \text{in\;} Q, \quad \forall i \in \{0,\dots, N\}.
\ee
In the above relation, $\g$ is the gravity vector, $\mu_i$ denotes the constant viscosity of the phase $i$, $p_i$ its pressure, and $\rho_i$ its density. 
The intrinsic permeability tensor $\bbK: \ov \O \to \R^{d\times d}$ is supposed to be smooth, symmetric
$\bbK = \bbK^T$, 
and uniformly positive definite: there exist $\kappa_\star, \kappa^\star > 0$ such that:
\be\label{eq:Kelliptic}
\kappa_\star |\bxi|^2 \le \bbK(\x) \bxi\cdot\bxi \le \kappa^\star |\bxi|^2, \qquad \forall \bxi \in \R^d, 
\; \forall \x \in \ov \O.
\ee
The pore volume is supposed to be saturated by the fluid mixture
\be\label{eq:sature} 
\sigma:= \sum_{i=0}^N s_i = \omega(\x) \quad \text{a.e. in}\; Q, 
\ee
where the porosity  $\omega: \ov\O \to (0,1)$ of the {surrounding porous matrix} is assumed to be smooth.
In particular, there exists $0<\omega_\star\leq \omega^\star$ such that $\omega_\star\leq \omega(\x) \leq \omega^\star$ for all $\x \in \ov \O$.
In what follows, we denote by 
$\bs = (s_0, \dots, s_N)$, by
$$
\Delta(\x) = \left\{\bs \in (\R_+)^{N+1} \middle| \; \sum_{i=0}^N s_i = \omega(\x) \right\}, \; 
$$
and by 
$$
\bXx = \left\{ \bs \in L^1(\O; \R_+^{N+1})\; \middle| \; 
\bs(\x) \in \Delta(\x) \; \text{a.e. in}\; \O \right\}.
$$
There is an obvious one-to-one mapping between the sets $\Delta(\x)$ and 
$$\Delta^*(\x) = \left\{\bs^\ast = (s_1, \dots, s_N) \in (\R_+)^{N}  \middle| \; 
\sum_{i=1}^N s_i \le \omega(\x) \right\}, $$
and consequently also between $\bXx$ and 
$$
\bXx^\ast = \left\{ \bs^\ast \in L^1(\O; \R_+^{N})\; \middle| \; 
\bs^*(\x) \in \Delta^*(\x) \; \text{a.e. in}\; \O \right\}.
$$
In what follows, we denote by 
$\bUpsilon = \bigcup\limits_{\x \in \ov \O} \Delta^*(\x) \times \{\x\}.$

In order to close the system,  we impose $N$ capillary pressure relations
\be\label{eq:pi_i}
p_i - p_0 = \pi_i(\bs^*, \x)\quad \text{a.e in}\; Q, \quad \forall i \in \{1,\dots,N\}, 
\ee
where the capillary pressure functions 
$\pi_i: \bUpsilon  \to \R$ are assumed to be continuously differentiable 
and to derive from a strictly convex potential $\Pi: \bUpsilon \to \R_+$:
\begin{equation*}
\label{eq:pi-pot}
\pi_i(\bs^*,\x) = \frac{\p\Pi}{\p s_i}(\bs^*,\x) 
\qquad \forall i \in \{1,\dots, N\}.
\end{equation*}
We assume that $\Pi$ is uniformly convex w.r.t. its first variable. 
More precisely, we assume that there exist two positive constants $\varpi_\star$ and $\varpi^\star$ 
such that, for all $\x \in \ov \O$ and all $\bs^*, \wh \bs^* \in \Delta^*(\x)$, one has 
\be\label{eq:Pi-unif-conv}
\frac{\varpi^\star}{2} {|\wh \bs^* - \bs^*|}^2 \ge 
\Pi(\wh \bs^*,\x) - \Pi(\bs^*,\x) - \bpi(\bs^*,\x)\cdot(\wh \bs^* - \bs^*)
\ge \frac{\varpi_\star}{2} {|\wh \bs^* - \bs^*|}^2, 
\ee
where we introduced the notation
$$
\bpi: \begin{cases}
\bUpsilon \to \R^N\\
(\bs^*,\x) \mapsto \bpi(\bs^*,\x) = \left(\pi_1(\bs^*,\x), \dots,\pi_N(\bs^*,\x)\right).
\end{cases}
$$
The relation~\eqref{eq:Pi-unif-conv} implies that $\bpi$
is monotone and injective w.r.t. its first variable.
Denoting by
$$
\bz\mapsto \boldsymbol\phi(\bz,\x)=(\phi_1(\bz,\x),\dots,\phi_N(\bz,\x)) \in \Delta^*(\x)
$$
the inverse of $\bpi(\cdot,\x)$, it follows from~\eqref{eq:Pi-unif-conv} that
\be
\label{eq:H_pi_1}
0<\frac1{\varpi^\star}\leq \bbJ_\bz \boldsymbol\phi(\bz,\x)\leq \frac1{\varpi_\star}
\qquad \text{for all }\x\in\ov\O\text{ and all }\bz\in \bpi(\Delta^*(\x),\x),
\ee
where $\bbJ_\bz$ stands for the Jacobian with respect to $\bz$ and 
the above inequality should be understood in the sense of positive definite 
matrices. 
Moreover, due to the regularity of $\bpi$ w.r.t. the space variable, there 
exists $M_\bphi>0$ such that 
\be
\label{eq:H_pi_2}
|\grad_\x \boldsymbol \phi (\bz,\x)|\leq M_\bphi
\qquad \text{for all }\x\in\ov\O\text{ and all }\bz\in \bpi(\Delta^*(\x),\x),
\ee
where $\grad_\x$ denote the gradient w.r.t. to the second variable only.
\smallskip

The problem is complemented with no-flux boundary conditions
\be\label{eq:no-flux}
\bv_i \cdot \n = 0 \quad \text{on} \; \p\O \times (0,T), \quad \forall i \in \{0,\dots, N\},
\ee
and by the initial content profile $\bs^0 = \left(s_0^0, \dots, s_N^0\right) \in \bXx$:
\be\label{eq:init}
s_i(\cdot, 0) = s_i^0 \quad \forall i \in \{0,\dots, N\}, \quad \text{with}\quad
\sum_{i=0}^N s_i^0 = \omega \; \text{a.e. in}\ \O.
\ee

Since we did not consider sources, and since we imposed no-flux boundary conditions, 
the volume of each phase is conserved along time
\be\label{eq:mi}
\int_\O s_i(\x,t) \d\x = \int_\O  s_i^0(\x) \d\x =:m_i>0, \qquad \forall i \in \{0,\dots, N\}.
\ee

We can now give a proper definition of what we call a weak solution to the problem
\eqref{eq:cons-mass}--\eqref{eq:Darcy}, \eqref{eq:sature}--\eqref{eq:pi_i}, and
\eqref{eq:no-flux}--\eqref{eq:init}.
\begin{Def}[Weak solution]\label{Def:weak}
A measurable function $\bs: Q \to (\R_+)^{N+1}$ is said to be a weak solution if 
$\bs \in \Delta$ a.e. in $Q$, if there exists $\bp = (p_0, \dots, p_N) \in L^2((0,T);H^1(\O))^{N+1}$ 
such that the relations~\eqref{eq:pi_i} hold, and such that, for all 
$\phi \in C^\infty_c(\ov \O \times [0,T))$ and all $i \in \{0,\dots, N\}$, one has  
\be\label{eq:weak}
\iint_Q  s_i \p_t \phi \d\x \d t + \int_\O s_i^0 \phi(\cdot, 0) \d\x 
- \iint_Q \frac{s_i}{\mu_i} \bbK \left(\grad p_i - \rho_i \g\right)\cdot \grad \phi \d\x \d t = 0.
\ee
\end{Def}

\subsection{Wasserstein gradient flow of the energy}\label{ssec:GFW}

\subsubsection{Energy of a configuration}\label{sssec:energy}

First, we extend the convex function $\Pi: \bUpsilon \to [0,+\infty]$, called 
\emph{capillary energy density}, to a convex function 
(still denoted by) $\Pi: \R^{N+1}\times \ov \O \to [0,+\infty]$ by setting 
$$
\Pi(\bs,\x) = \begin{cases}
\Pi\left(\omega \frac{\bs^*}{\sigma},\x\right) = \Pi\left(\omega \frac{s_1}{\sigma}, \dots, 
\omega\frac{s_N}{\sigma},\x\right) & \text{if } \bs \in \R_+^{N+1} \; \text{and}\;  \sigma \le \omega(\x), \\
+ \infty & \text{otherwise},
\end{cases}
$$
$\sigma$ being defined by~\eqref{eq:sature}. The extension of $\Pi$ by $+\infty$ 
where $\sigma >\omega$ is natural because of the incompressibility of the fluid mixture.
The extension to $\{\sigma <\omega\}\cup \R_+^{N+1}$ is designed so that the 
energy density 
only depends on the relative composition of the fluid mixture. However, this extension is somehow 
arbitrary, and, as it will appear in the sequel, it has no influence on the flow since 
the solution $\bs$ remains in $\bXx$ {(i-e $\sum_{i = 0}^N s_i=\omega$)}.
In our previous note~\cite{CGM15} the appearance of void $\sigma<\omega$ was directly
prohibited by a penalization in the energy.
\smallskip

The second part in the energy comes from the gravity. 
In order to lighten the notations, we introduce the functions 
\begin{equation*}
\label{eq:Psi_i}
\Psi_i: \left\{\begin{array}{rcl}
\ov \O &\to& \R_+, \\
\x &\mapsto & - \rho_i \g\cdot \x, 
\end{array}\right.
\qquad \forall i \in \{0,\dots, N\}, 
\end{equation*}
and
$$\bPsi:
\left\{\begin{array}{rcl}
\ov \O &\to& \R_+^{N+1}, \\
\x &\mapsto & \left(\Psi_0(\x), \dots, \Psi_N(\x)\right).
\end{array}\right.
$$
The fact that $\Psi_i$ can be supposed to be positive come from the fact that 
$\O$ is bounded.
Even though the physically relevant potentials are indeed the gravitational $\Psi_i(\x)=-\rho_i \g\cdot \x$, the subsequent analysis allows for a broader class of external potentials and for the sake of generality we shall therefore consider arbitrary $\Psi_i\in\mathcal C^1(\overline{\Omega})$ in the sequel.
\smallskip

We can now define the convex energy functional $\Ee:  L^1(\O,\R^{N+1}) \to \R\cup\{+\infty\}$ 
by adding the capillary energy to the gravitational one:
\be\label{eq:Ee}
\Ee(\bs) = \int_\O \left(\Pi(\bs,\x) + \bs\cdot \bPsi\right)  \d\x \ge 0, 
\qquad \forall \bs \in L^1(\O; \R^{N+1}).
\ee
Note moreover that $\Ee(\bs)<\infty$ iff $\bs \ge 0$ and $ \sigma \le \omega$ a.e. in $\O$.
It follows from the mass conservation~\eqref{eq:mi} that 
$$
\int_\O \sigma(\x) \d\x = \sum_{i=0}^N m_i = \int_\O \omega(\x) \d\x. 
$$
Assume that there exists a non-negligible subset $A$ of $\O$ such that $\sig < \omega$ on $A$, 
then necessarily, there must be a non-negligible subset $B$ of $\O$ such that $\sig > \omega$ 
so that the above equation holds, hence $\Ee(\bs) = +\infty$. 
Therefore, 
\be
\label{eq:E-finie}
\Ee(\bs) < \infty \quad \Leftrightarrow \quad \bs \in \bXx.
\ee

Let $\bp = (p_0, \dots, p_N): \O \to \R^{N+1}$  be such that 
$\bp \in \p_\bs\Pi(\bs,\x)$ for a.e. $\x$ in $\O$, then, defining 
$h_i = p_i + \Psi_i(\x)$ for all $i \in \{0,\dots, N\}$ and $\bh = \left(h_i\right)_{0\le i \le N}$, 
$\bh$ belongs to the subdifferential $\p_\bs \Ee(\bs)$ of $\Ee$ at $\bs$, i.e.,
$$
\Ee(\wh \bs) \ge \Ee(\bs) + \sum_{i=0}^N \int_\O  h_i (\wh s_i - s_i) \d\x, \qquad 
\forall \wh \bs \in L^1(\O;\R^{N+1}).
$$
The reverse inclusion also holds, hence
\be\label{eq:dE}
\p_\bs\Ee(\bs) = \left\{ \bh:\O \to \R^{N+1}\; \middle| \; 
h_i - \Psi_i(\x) \in \p_\bs \Pi(\bs,\x) \;\text{for a.e.}\; \x \in \O\right\}.
\ee

Thanks to~\eqref{eq:E-finie}, we know that a configuration $\bs$ has finite energy 
iff $\bs\in\bXx$. 
Since we are interested in finite energy configurations, it is relevant to consider the 
restriction of $\Ee$ to $\bXx$.
Then using the one-to-one mapping between $\bXx$ and $\bXx^\ast$, 
we define the energy of a configuration $\bs^*\in \bXx^\ast$, that we denote by $\Ee(\bs^*)$ 
by setting $\Ee(\bs^*) = \Ee(\bs)$ where $\bs$ is the unique element of $\bXx$ corresponding 
to $\bs^*\in\bXx^*$.

\subsubsection{Geometry of $\O$ and Wasserstein distance}\label{sssec:W2}

Inspired by the paper of Lisini~\cite{Lisini09}, where heterogeneous anisotropic 
degenerate parabolic equations are studied from a variational point of view, 
we introduce $(N+1)$ distances on $\O$ that take into account the permeability of the porous medium 
and the phase viscosities.
Given two points $\x,\by$ in $\O$, we denote by 
$$
P(\x,\by) = \left\{\bga \in C^1([0,1]; \O)\, \middle| \, \bga(0) = \x\;  \text{and}\; 
\bga(1) = \by\right\} $$
the set of the smooth paths joining $\x$ to $\by$, 
and we introduce distances $d_i$, $ i \in \{0,\dots, N\}$ between elements on $\O$ by setting
\begin{equation}
\label{eq:def_d_i}
d_i(\x,\by) = \inf_{\bga \in P(\x,\by)}\left( \int_0^1 \mu_i\bbK^{-1}(\bga(\tau)) 
\bga'(\tau) \cdot \bga'(\tau) \d \tau \right)^{1/2}, \qquad \forall (\x,\by)\in \ov \O.
\end{equation}
It follows from~\eqref{eq:Kelliptic} that 
\be\label{eq:d_i-euclid}
\sqrt{\frac{\mu_i}{\kappa^\star}}|\x-\by| \le d_i(\x,\by) \le \sqrt{\frac{\mu_i}{\kappa_\star}} |\x-\by|, 
\qquad \forall (\x,\by)\in \ov \O^2.
\ee

For $i \in \{0,\dots, N\}$ we define 
$$\Aa_i = \left\{ s_i \in L^1(\O; \R_+) \, \middle| \, \int_\O s_i \d\x = m_i \right\}.$$
Given $s_{i}, \wh s_{i}\in\Aa_i$, the set of admissible transport plans between $s_i$ and $\wh s_i$ 
is given by
$$
\Gamma_i(s_{i}, \wh s_{i}) = \left\{ \theta_i \in \Mm_+(\O\times\O)\, \middle| 
\,\theta_i(\Omega\times\Omega)=m_i,\, 
\theta_i^{(1)} = s_i \; \text{and} \; \theta_i^{(2)} = \wh s_i \;
\right\},
$$
where $\Mm_+(\O\times\O)$ stands for the set of Borel measures on $\O\times \O$ 
and $\theta_i^{(k)}$ is the $k^\text{th}$ marginal of the measure $\theta_i$.
We define the quadratic Wasserstein distance $W_i$ on $\Aa_i$ by setting
\be\label{eq:Wi}
W_i(s_{i}, \wh s_{i}) = \left(\inf_{\theta_i \in \Gamma(s_{i}, \wh s_{i})} \iint_{\O\times\O} d_i(\x,\by)^2 \d \theta_i(\x,\by)\right)^{1/2}.
\ee
Due to the permeability tensor $\mathbb K(\x)$, the porous medium $\O$ might be heterogeneous and anisotropic. 
Therefore, some directions and areas might me privileged by the fluid motions. 
This is encoded in the distances $d_i$ we put on $\O$. 
Moreover, the more viscous the phase is, the more costly are its displacements, 
hence the $\mu_i$ in the definition~\eqref{eq:def_d_i} of $d_i$. But it follows from~\eqref{eq:d_i-euclid} 
that 
\be\label{eq:W_i-equiv}
\sqrt{\frac{\mu_i}{\kappa^\star}} W_{\rm ref}(s_i,\wh s_i) \le
W_{i}(s_i,\wh s_i)
\le \sqrt{\frac{\mu_i}{\kappa_\star}} W_{\rm ref}(s_i,\wh s_i)., 
\qquad \forall s_i, \wh s_i \in \Aa_i, 
\ee
where $W_{\rm ref}$ denotes the classical quadratic Wasserstein distance defined by 
\be\label{eq:Wref}
W_{\rm ref}(s_i,\wh s_i) = 
 \left(\inf_{\theta_i \in \Gamma(s_{i}, \wh s_{i})} \iint_{\O\times\O} |\x-\by|^2 \d \theta_i(\x,\by)\right)^{1/2}.
\ee

\smallskip

With the phase Wasserstein distances $\left(W_i\right)_{0\le i \le N}$ at hand, 
we can define the global Wasserstein distance $\bW$ on 
$\bAa:=\Aa_0 \times \dots \times \Aa_N$ by setting 
$$
\bW(\bs, \wh \bs) = \left( \sum_{i=0}^N W_i(s_i, \wh s_i)^2\right)^{1/2}, \quad \forall \bs, \wh \bs \in \boldsymbol{\mathcal A}.
$$

{
Finally for technical reasons we also assume that there exist smooth extensions $ \widetilde \bbK$ and $ \widetilde \omega$ to $\R^d$ of the tensor and the porosity, respectively, such that \eqref{eq:Kelliptic} holds on $\R^d$ for $ \widetilde \bbK$,  and such that $ \tilde \omega$ is strictly bounded from below. This allows to define distances $\wt d_i$ on the whole $\R^d$by 
\begin{equation}
\label{eq:def_d_i-tilde}
\wt d_i(\x,\by) = \inf_{\bga \in \wt P(\x,\by)}\left( \int_0^1 \mu_i\wt \bbK^{-1}(\bga(\tau)) 
\bga'(\tau) \cdot \bga'(\tau) \d \tau \right)^{1/2}, \qquad \forall \x,\by\in \R^d
\end{equation}
where 
$
\wt P(\x,\by) = \left\{\bga \in C^1([0,1]; \R^d)\, \middle| \, \bga(0) = \x\;  \text{and}\; 
\bga(1) = \by\right\}.
$
In the sequel, we assume that the extension $\wt \bbK$ of $\bbK$ is such that 
 \begin{equation}
\label{eq:hyp_Omega_convex}
\O \text{ is geodesically convex in } \mathcal M_i = (\R^d,  \widetilde d_i )  \text{ for all }i.
\end{equation}
In particular $\widetilde d_i =d_i$ on $\O\times \O $. Since $\widetilde \bbK^{-1}$ is smooth, at least $C_b^2(\R^d)$, the Ricci curvature of the smooth complete Riemannian manifold $ \mathcal M_i$ is 
uniformly bounded, i.e., there exists $C$ depending only on $\left(\mu_i\right)_{0\le i \le N}$ 
and $\widetilde \bbK$ such that 
\be\label{eq:Ricci}
| {\rm Ric}_{\mathcal M_i,\x} (\bv) | \le C  \mu_i \bbK^{-1}\bv \cdot \bv, 
\qquad 
\forall \x \in \R^d, \; \forall \bv \in \R^d.
\ee
Combined with the assumptions on $ \widetilde \omega$ we deduce that $\mathcal H_{\widetilde \omega}$ is $\widetilde \lambda_i$ displacement convex on $\mathcal P_2^{ac}(\mathcal M_i)$ for some $\widetilde \lambda_i \in \R$. 
Then
\eqref{eq:hyp_Omega_convex} and mass scaling implies that $\mathcal H_{ \omega}$ is $\lambda_i$ displacement convex on $( \mathcal A_i,W_i)$ for some $ \lambda_i \in \R$.
We refer to  \cite[Chap. 14 \& 17]{Villani09} for further details on the Ricci curvature and its links with
optimal transportation.

In the homogeneous and isotropic case $\bbK(\x)=\operatorname{Id}$, Condition~\eqref{eq:hyp_Omega_convex} simply amounts to assuming that $\Omega$ is convex.
A simple sufficient condition implying~\eqref{eq:hyp_Omega_convex} is given in Appendix~\ref{app:convex} in the isotropic but heterogeneous case $\bbK(\x) = \kappa(\x) \bbI_d$.}

\subsubsection{Gradient flow of the energy}\label{sssec:GF}

The content of this section is formal. 
Our aim is to write the problem as a gradient flow, i.e. 
\be\label{eq:GF}
\frac{\d\bs }{\d t}  \in - \textbf{grad}_{\bW} \Ee(\bs) = - \left(\text{grad}_{W_0} \Ee(\bs), \dots , \text{grad}_{W_N} \Ee(\bs)\right)
\ee
where $\textbf{grad}_\bW  \Ee(\bs)$ denotes the full Wasserstein gradient of $\Ee(\bs)$, 
and $\text{grad}_{W_i}\Ee(\bs)$ stands for the partial gradient of $s_i \mapsto \Ee(\bs)$ with respect to the Wasserstein distance $W_i$.
The Wasserstein distance $W_i$ was built so that 
$\dot \bs = \left(\dot s_i\right)_{i} \in \textbf{grad}_{\bW} \Ee(\bs)$ iff there exists $\bh \in \p_\bs\Ee(\bs)$ such that  
$$
\partial_t s_i = -  \div\left(s_i\frac{\bbK}{\mu_i}  \grad h_i   \right), \qquad \forall i \in \{0,\dots, N\}.$$
Such a construction was already performed by Lisini in the case of a single equation.
Owing to the definitions~\eqref{eq:Ee} and \eqref{eq:dE} of the energy $\Ee(\bs)$ and its subdifferential $\p_\bs\Ee(\bs)$, the partial differential 
equations can be (at least formally) recovered. This was roughly speaking to purpose of our note~\cite{CGM15}.
\smallskip

In order to define rigorously the gradient $\textbf{grad}_\bW \Ee$ in~\eqref{eq:GF}, $\bAa$ has to be a Riemannian manifold. 
The so-called Otto's calculus (see \cite{Otto01} and~\cite[Chapter 15]{Villani09}) 
allows to put a formal Riemannian structure on $\bAa$. But as far as we know, this structure 
cannot be made rigorous and $\bAa$ is a mere metric space. 
This leads us to consider generalized gradient flows in metric spaces (cf. \cite{AGS08}).
We won't go deep into details in this direction, but we will prove that weak solutions 
can be obtained as limits of a minimizing movement scheme presented in the next section. 
This characterizes the gradient flow structure of the problem.

\subsection{Minimizing movement scheme and main result}\label{ssec:JKO}

\subsubsection{The scheme and existence of a solution}\label{sssec:scheme}
For a fixed time-step $\tau>0$, the so-called minimizing movement scheme~\cite{DeGiorgi93, AGS08}
or JKO scheme~\cite{JKO98} consists in computing recursively $\left(\bs^n\right)_{n\ge 1}$ as 
the solution to the minimization problem  
\be\label{eq:JKO}
\bs^{n} = \underset{\bs\in \bAa}{\operatorname{Argmin}}\left(\frac{\bW(\bs, \bs^{n-1})^2}{2\tau} + \Ee(\bs) \right), 
\ee
the initial data $\bs^0$ being given \eqref{eq:init}.

\subsubsection{Approximate solution and main result}\label{sssec:main}
Anticipating that the JKO scheme \eqref{eq:JKO} is well posed (this is the purpose of Proposition~\ref{prop:exist-JKO} below), we can now define the 
piecewise constant interpolation
$\bs^\tau \in L^\infty((0,T);\bXx\cap \bAa)$ by 
\be\label{eq:stau}
\bs^\tau(0,\cdot) = \bs^0, \quad \text{and}\quad 
\bs^\tau(t,\cdot) = \bs^n \quad\forall  t \in ((n-1)\tau,n\tau], \, \forall n \ge 1.
\ee
The main result of our paper is the following. 
\begin{thm}\label{thm:main}
Let $\left(\tau_k\right)_{k\ge 1}$ be a sequence of time steps tending to $0$, 
then there exists one weak solution $\bs$ in the sense of Definition~\ref{Def:weak} 
such that, up to an unlabeled subsequence, $\left(\bs^{\tau_k}\right)_{k\ge 1}$ converges a.e. 
in $Q$ towards $s$ as $k$ tends to $\infty$. 
\end{thm}
As a direct by-product of Theorem~\ref{thm:main}, the continuous problem admits 
(at least) one solution in the sense of Definition~\ref{Def:weak}.
As far as we know, this existence result is new.

\begin{rem}
It is worth stressing that our final solution will satisfy {\em a posteriori} 
$\partial_t s_i\in L^2((0,T);H^1(\Omega)')$,  $s_i \in L^2((0,T);H^1(\O))$, and thus
$s_i\in \mathcal C([0,T];L^2(\O))$. This regularity is enough to retrieve the so-called 
\emph{Energy-Dissipation-Equality}
$$
\frac{d}{dt}\mathcal E(\bs(t))=-\sum_{i=0}^N \int_\Omega\bbK\frac{s_i(t)}{\mu_i}\nabla (p_i(t)+\Psi_i)\cdot\nabla (p_i(t)+\Psi_i) \d\x \le 0 \quad \text{for a.e. }t\in (0,T),
$$
which is another admissible formulation of gradient flows in metric spaces \cite{AGS08}.
\end{rem}

\subsection{Goal and positioning of the paper}\label{ssec:positioning}

The aims of the paper are twofolds.
First, we aim to provide rigorous foundations 
to the formal variational approach exposed in the authors' recent note~\cite{CGM15}. 
This gives new 
insights into the modeling of complex porous media flows and their numerical 
approximation. Our approach appears to be very natural since only physically motivated 
quantities appear in the study. Indeed, we manage to avoid the introduction 
of the so-called Kirchhoff transform and global pressure, which classically appear in 
the mathematical study of multiphase flows in porous media (see for instance 
\cite{Chavent76, AM78, CJ86, FS93, GMT96, Chen01, Chavent09, AJV12, AJZK14}). 
\smallskip

Second, the existence result that we deduce from the convergence 
of the variational scheme is new as soon as there are at least three 
phases ($N \ge 2$). Indeed, since our study does not require the 
introduction of any global pressure, we get rid of many structural 
assumptions on the data among which the so-called \emph{total 
differentiability condition},  see for instance Assumption~{\bf (H3)} 
in the paper by Fabrie and Saad~\cite{FS93}. This structural condition is not naturally satisfied 
by the models, and suitable algorithms have to be employed in order 
to adapt the data to this constraint~\cite{GS85}. However, our approach 
suffers from another technical difficulty: we are stuck to the case 
of linear relative permeabilities. The extension to the case of nonlinear 
concave relative permeabilities, i.e., where~\eqref{eq:cons-mass} is replaced by 
$$
\p_t s_i + \div (k_i(s_i) \bv_i) = 0,
$$
may be reachable thanks to the contributions of Dolbeault, Nazaret, and 
Savar\'e~\cite{DNS09} {(see also \cite{ZM15_mob}}), but we did not push in this direction since 
the relative permeabilities $k_i$ are in general supposed to be convex
in models coming from engineering. 
\smallskip

Since the seminal paper of Jordan, Kinderlehrer, and Otto~\cite{JKO98}, 
gradient flows in metric spaces (and particularly in the space of probability measures endowed with 
the quadratic Wasserstein distance) were the object of many studies. Let us for 
instance refer to the monograph of Ambrosio, Gigli, and Savar\'e~\cite{AGS08} 
and to Villani's book~\cite[Part II]{Villani09} for a complete overview. Applications are numerous. 
We refer for instance to~\cite{Otto98} for an application to magnetic fluids, to~\cite{SS04, AS08, AMS11}
for applications to supra-conductivity, to~\cite{BCC08, Bla14_KS, ZM15} for applications to chemotaxis, 
to~\cite{LMS12} for phase field models, to~\cite{MRS10} for a macroscopic model of crowd motion, 
to~\cite{BGG13} for an application to granular media, to~\cite{CDFLS11} for aggregation equations,
or to~\cite{KMX17} for a model of ionic transport that applies in semi-conductors.
In the context of porous media flows, this framework has been used by Otto~\cite{Otto01} 
to study the asymptotic behavior of the porous medium equation, that is a simplified model 
for the filtration of a gas in a porous medium. The gradient flow approach in Wasserstein metric spaces 
was used more recently by Lauren\c{c}ot and Matioc~\cite{LM13_Muskat} on a thin film approximation 
model for two-phase flows in porous media.
{Finally, let us mention that similar ideas were successfully applied for multicomponent systems, see e.g. \cite{CarlierLaborde15,Laborde_PHD,ZM15_mob,zinsl14}.}

The variational structure of the system governing 
incompressible immiscible two-phase flows in porous media was recently depicted 
by the authors in their note~\cite{CGM15}. Whereas the purpose of~\cite{CGM15} is formal, our 
goal is here to give a rigorous foundation to the variational approach for complex flows in porous 
media. Finally, let us mention the work of Gigli and Otto~\cite{GO13} where it was noticed that 
multiphase linear transportation with saturation constraint (as we have here thanks 
to~\eqref{eq:cons-mass} and~\eqref{eq:sature}) yields nonlinear transport with mobilities that appear 
naturally in the two-phase flow context. 
\smallskip

The paper is organized as follows.
In Section~\ref{sec:one-step}, we derive estimates on the solution $\bs^\tau$ for a fixed $\tau$.
Beyond the classical energy and distance estimates detailed 
in \S\ref{ssec:energy}, we obtain enhanced regularity estimates thanks to an adaptation of the so-called 
{\em flow interchange} technique of Matthes, McCann, and Savar\'e~\cite{matthes2009family} 
to our inhomogeneous context in~\S\ref{ssec:flow-interchange}. 
Because of the constraint on the pore volume~\eqref{eq:sature}, the auxiliary flow we use is 
no longer the heat flow, and a drift term has to be added. An important effort is then done 
in \S\ref{ssec:Euler-Lagrange} to derive the Euler-Lagrange equations that follow from the 
optimality of $\bs^n$. Our proof is inspired from the work of Maury, Roudneff-Chupin, and 
Santambrogio~\cite{MRS10}. It relies on an intensive use of the dual characterization of the 
optimal transportation problem and the corresponding Kantorovitch potentials.
However, additional difficulties arise from the multiphase aspect of our problem, in particular 
when there are at least three phases (i.e., $N\ge2$). 
These are overpassed using a {generalized multicomponent bathtub principle 
(Theorem~\ref{theo:bathtub} in Appendix)} and computing the associated Lagrange multipliers
in \S\ref{subsec:decomposition}. This key step then allows to define the notion of discrete 
phase and capillary pressures in \S\ref{subsec:pressures}. Then 
Section~\ref{sec:convergence} is devoted to the convergence of the approximate solutions 
$\left(\bs^{\tau_k}\right)_k$ towards a weak solution $\bs$ as $\tau_k$ tends to $0$. 
The estimates we obtained in Section~\ref{sec:one-step} are integrated w.r.t. time 
in~\S\ref{ssec:estimates-time}. In \S\ref{ssec:compact}, we show that these estimates 
are sufficient to enforce the relative compactness of $\left(\bs^{\tau_k}\right)_{k}$ in the 
strong $L^1(Q)^{N+1}$ topology. Finally, it is shown in \S\ref{ssec:identify} that any limit 
$\bs$ of $\left(\bs^{\tau_k}\right)_{k}$ is a weak solution in the sense of Definition~\ref{Def:weak}.

\section{One-step regularity estimates}\label{sec:one-step}
The first thing to do is to show that the JKO scheme \eqref{eq:JKO} is well-posed.
This is the purpose of the following Proposition.

\begin{prop}\label{prop:exist-JKO}
Let $n \ge 1$ and $\bs^{n-1} \in \bXx \cap \bAa$, then there exists a unique solution $\bs^n$ 
to the scheme~\eqref{eq:JKO}. Moreover, one has $\bs^n \in \bXx \cap \bAa$.
\end{prop}
\begin{proof}
Any $\bs^{n-1}\in \bXx \cap \bAa$ has finite energy thanks to~\eqref{eq:E-finie}. 
Let ${(\bs^{n,k})}_k\subset\bAa$ be a minimizing sequence in \eqref{eq:JKO}.
Testing $\bs^{n-1}$ in \eqref{eq:JKO} it is easy to see that $\Ee(\bs^{n,k}) \le \Ee(\bs^{n-1}) < \infty$ for large $k$, thus ${(\bs^{n,k})}_k \subset \bXx\cap \bAa$ thanks to~\eqref{eq:E-finie}.  
Hence, one has $0\leq s^{n,k}_i(\x)\leq \omega(\x)$ for all $k$.
By Dunford-Pettis theorem, we can therefore assume that 
$s_i^{n,k}\rightharpoonup s_i^n$ weakly in $L^1(\Omega)$.
It is then easy to check that the limit $\bs^n$ of $\bs^{n,k}$ belongs to $\bXx\cap \bAa$.
The lower semi-continuity of the Wasserstein distance with respect to weak $L^1$ convergence is well known (see,  e.g.,  \cite[Prop. 7.4]{Santambrogio_OTAM}), and since the energy functional is convex thus l.s.c., we conclude that $\bs^n$ is indeed a minimizer.
Uniqueness follows from the strict convexity of the energy as well as from the convexity of the Wasserstein distances (w.r.t. linear interpolation $\bs_\theta=(1-\theta)\bs_0+\theta\bs_1$).
\end{proof}

The rest of this section is devoted to improving the regularity of the successive minimizers.

\subsection{Energy and distance estimates}\label{ssec:energy}
Testing $\bs=\bs^{n-1}$ in \eqref{eq:JKO} we obtain
\be
\label{eq:square_distance_one_step}
\frac{\bW(\bs^n, \bs^{n-1})^2}{2\tau} + \Ee(\bs^n)\leq \Ee(\bs^{n-1}),
\ee
As a consequence we have the monotonicity
$$
\ldots \leq \Ee(\bs^{n})\leq \Ee(\bs^{n-1})\leq \ldots \leq \Ee(\bs^0)<\infty
$$
at the discrete level, thus $\bs^n \in \bXx$ for all $n \ge 0$ thanks to~\eqref{eq:E-finie}.
Summing \eqref{eq:square_distance_one_step} over $n$ we also obtain the classical \emph{total square distance} estimate
\be
\label{eq:tot_square_dist}
\frac{1}{\tau}\sum\limits_{n\geq 0}\bW^2(\bs^{n+1},\bs^n)\; {\leq}\; 2\Ee(\bs^0) 
\le C\left(\O, \Pi, \bPsi \right), 
\ee
the last inequality coming from the fact that $\bs^0$ is uniformly bounded since it belongs to $\bXx$, 
thus so is $\Ee(\bs^0)$.
This readily gives the approximate $1/2$-H\"older estimate 
\be
\label{eq:discrete_12_holder}
\bW(\bs^{n_1},\bs^{n_2})\leq C\sqrt{|n_2-n_1|\tau}.
\ee

\subsection{Flow interchange, entropy estimate and enhanced regularity}\label{ssec:flow-interchange}

The goal of this section is to obtain some additional Sobolev regularity on the capillary pressure 
field $\bpi(\bs^{n*}, \x)$, where $\bs^{n*}=(s_1^n,\dots,s_N^n)$ is the unique element of $\bXx^*$ corresponding 
to the minimizer $\bs^n$ of~\eqref{eq:JKO}. 
In what follows, we denote by 
$$
\pi_i^n: \left\{\begin{array}{rcl}
\O & \to & \R, \\
\x &\mapsto& \pi_i(\bs^{n*}(\x), \x), 
\end{array}\right. 
\qquad \forall i \in \{1,\dots, N\}
$$
and $\bpi^n = \left(\pi_1^n, \dots, \pi_N^n\right)$.
Bearing in mind that $\omega(\x)\geq \omega_\star>0$ in $\overline\Omega$, we can define 
the relative Boltzmann entropy $\Hh_\omega$ with respect to $\omega$ by 
$$
\mathcal H_\omega(s):=\int_\Omega s(\x)\log\left(\frac{s(\x)}{\omega(\x)}\right)d \x,
\quad \text{for all measurable}\; s:\O \to \R_+.
$$
\begin{lem}
\label{lem:flow_interchange_one_step}
There exists $C$ depending only on $\O, \Pi, \omega, \bbK,
{(\mu_i)}_i,$ and $\bPsi$
such that, for all $n\ge 1$ and all $\tau>0$, 
one has 
 \be
 \label{eq:flow_interchange_one_step}
\sum \limits_{i=0}^N\|\nabla \pi_i^n\|^2_{L^2(\Omega)}
\leq C\left(1 + \frac{\bW^2(\bs^n,\bs^{n-1})}{\tau} + \sum\limits_{i=0}^N \frac{\mathcal H_\omega(s_i^{n-1})-\mathcal H_\omega(s_i^{n})}{\tau}\right).
 \ee
\end{lem}
\begin{proof}
 The argument relies on the \emph{flow interchange} technique introduced by Matthes, McCann, and Savar\'e 
 in \cite{matthes2009family}. Throughout the proof, $C$ denotes a fluctuating constant that depends on 
 the prescribed data $\O, \Pi, \omega, \bbK, {(\mu_i)}_i,$ and $\bPsi$, but neither on $t$, $\tau$, nor on $n$.
 For $i=0\ldots N$ consider the auxiliary flows
 \be \label{eq:auxiliary_flow_interchange}
 \left\{
 \begin{array}{ll}
{\p_t \check s_i}=\dive(\mathbb K\nabla \check s_i-\check s_i\mathbb K \nabla\log\omega),	& t>0,\,\x\in\Omega,\\
  \mathbb K(\nabla \check s_i-\check s_i \nabla\log\omega)\cdot \nu =0, & t>0,\,\x\in\p\Omega,\\
  \left. \check s_i\right|_{t=0}=s_i^n, & \x\in\Omega
 \end{array}
 \right.
 \ee
 for each $i\in\{0,\dots,N\}$.
 By standard parabolic theory (see for instance \cite[Chapter III, Theorem 12.2]{LSU68}), 
 these Initial-Boundary value problems are well-posed, and their solutions $\check s_i(\x)$ belong to 
 $\mathcal C^{1,2}((0,1]\times\overline\Omega) \cap \Cc([0,1];L^p(\O))$ for all $p\in(1,\infty)$ if $\omega \in \Cc^{2,\a}(\overline\Omega)$ and $\mathbb K\in\mathcal C^{1,\a}(\overline\Omega)$ for some $\alpha >0$.
Therefore, $t\mapsto \check s_i(\cdot, t)$ is absolutely continuous in $L^1(\O)$, thus in $\Aa_i$ endowed with 
the usual quadratic distance $W_{\rm ref}$~\eqref{eq:Wref} thanks to \cite[Prop. 7.4]{Santambrogio_OTAM}. 
Because of~\eqref{eq:W_i-equiv}, the curve $t\mapsto \check s_i(\cdot, t)$ is also 
absolutely continuous in $\Aa_i$ endowed with $W_i$.
 
 From Lisini's results \cite{Lisini09}, we know that the evolution $t\mapsto \check s_i(\cdot, t)$ can be interpreted as the gradient flow of the relative Boltzmann functional $\frac{1}{\mu_i}\mathcal H_\omega$ with respect to the metric $W_i$, the scaling factor $\frac{1}{\mu_i}$ appearing due to the definition \eqref{eq:Wi} of the distance $W_i$.
As a consequence of~\eqref{eq:Ricci}, The Ricci curvature of $(\O,d_i)$ is bounded, hence bounded from below. 
 Since $\omega\in \mathcal C^2(\overline\Omega)$ {and with our assumption \eqref{eq:hyp_Omega_convex}} we also have that $\frac{1}{\mu_i}\mathcal H_\omega$ is $\lambda_i$-displacement convex with respect to $W_i$ for some $\lambda_i\in \R$ depending on $\omega$ and the geometry of $(\Omega,d_i)$, see \cite[Chapter 14]{Villani09}.
 Therefore, we can use the so-called \emph{Evolution Variational Inequality} characterization of gradient flows (see for instance \cite[Definition 4.5]{AG13_UGOT}) centered at $s_i^{n-1}$, namely
\begin{equation*}
\label{eq:EVI-i}
 \frac{1}{2}\frac{\d}{\d t}W_i^2(\check s_i(t),s_i^{n-1})+\frac{\lambda_i}{2}W_i^2(\check s_i(t),s^{n-1}_i)\leq \frac{1}{\mu_i}\mathcal H_\omega(s_i^{n-1})-\frac{1}{\mu_i}\mathcal H_\omega(\check s_i(t)).
\end{equation*}
 Denote by $\check \bs = (\check s_0, \dots, \check s_N)$, and by 
 $\check \bs^* = (\check s_1, \dots, \check s_N)$. 
 Summing the previous inequality over $i\in\{0,\dots,N\}$ leads to
 \be \label{eq:inter_estimate_W}
\frac{\d}{\d t}\left(\frac{1}{2\tau}\bW^2(\check\bs(t),\bs^{n-1})\right)\leq 
C\left(\frac{\bW^2(\check\bs(t),\bs^{n-1})}{\tau} + \sum\limits_{i=0}^N 
\frac{\mathcal H_\omega(s_i^{n-1})-\mathcal H_\omega(\check s_i(t))}{\tau}\right).
 \ee
 
 In order to estimate the internal energy contribution in \eqref{eq:JKO}, we first note that 
 $\sum s^n_i(\x)=\omega(\x)$ for all $\x \in \ov \O$, thus by linearity 
 of \eqref{eq:auxiliary_flow_interchange} and since $\omega$ is a stationary solution we have $\sum \check s_i(\x,t)=\omega(\x)$ as well. 
 Moreover, the problem~\eqref{eq:auxiliary_flow_interchange} is monotone, thus order preserving, 
 and admits $0$ as a subsolution. Hence $\check s_i(\x,t)\geq 0$, so that $\check \bs(t)\in 
 \bAa\cap\bXx$ is an admissible competitor in \eqref{eq:JKO} for all $t>0$.
The smoothness of $\check \bs$ for $t>0$ allows to write 
 \be\label{eq:capi-interchange}
 \frac{\d}{\d t}\left(\int_\Omega \Pi(\check \bs^*(\x,t),\x) \d\x\right)=
 \sum_{i=1}^N\int_\Omega \check \pi_i(\x,t)
 \p_t \check s_i(\x,t)\d\x = I_1(t) + I_2(t),
 \ee
 where $\check \pi_i := \pi_i(\check \bs^*, \cdot)$, and where, for all $t>0$,  we have set 
$$
 I_1(t) =  -\sum_{i=1}^N\int_\Omega \grad \check \pi_i(t) \cdot \bbK\grad \check s_i(t) \d\x, \quad 
 I_2(t) = - \sum_{i=1}^N\int_\Omega \frac{\check s_i(t)}{\omega}\grad \check \pi_i(t) \cdot \bbK \grad \omega \d\x. 
$$
To estimate $I_1$, we first use the invertibility of $\bpi$ 
to write 
$$\check \bs(\x,t)=\bphi(\check \bpi(\x,t),\x)=: \check \bphi(\x,t),$$
yielding
\be
\label{eq:nabla_s_nabla_pi}
\grad \check \bs(\x,t)=\bbJ_{\bz}\boldsymbol{\phi}(\check \bpi(\x,t),\x)\grad \check \bpi(\x,t)
+ \grad_\x \boldsymbol{\phi}(\check \bpi(\x,t),\x).
\ee
Combining \eqref{eq:Kelliptic}, \eqref{eq:H_pi_1}, \eqref{eq:H_pi_2} and the elementary inequality 
\be\label{eq:Young-delta}
\displaystyle ab \le  \delta\frac{a^2}{2} + \frac{b^2}{2\delta} \quad\text{with $\delta>0$ arbitrary}, 
\ee
we get that for all $t>0$, there holds
$$
I_1(t)\leq - \frac{\kappa_\star}{\varpi^\star} \int_\Omega |\grad \check\bpi(t)|^2\d\x 
+ \kappa^\star \left(\delta\int_\Omega |\grad \check \bpi(t)|^2\d\x
+	\frac{1}{\delta}\int_\Omega|\grad_\x \boldsymbol{\phi}(\check\bpi(t))|^2\d\x\right).
$$
Choosing $\delta= \frac{\kappa_\star}{4\kappa^\star \varpi^\star}$, we get that 
\be\label{eq:I1}
I_1(t)\leq - \frac{3\kappa_\star}{4\varpi^\star} \int_\Omega |\grad \check\bpi(t)|^2\d\x + C
, \qquad \forall t >0.
\ee
\smallskip

In order to estimate $I_2$, we use that $\check \bs(t) \in \bXx$ for all $t>0$, so that 
$0\leq \check s_i(\x,t)\leq \omega(\x)$, hence we deduce that
$\sum_{i=1}^N \left(\frac{\check s_i}{\omega}\right)^2 \le 1.$
Therefore, using \eqref{eq:Young-delta} again, we get
$$
I_2(t) \leq  
 \delta \kappa^\star \int_\O |\grad \check \bpi(t)|^2 \d\x + \frac{\kappa^\star}{\delta} \int_\O |\grad \omega|^2\d\x.
$$
Choosing again $\delta= \frac{\kappa_\star}{4\kappa^\star \varpi^\star}$ yields 
\be\label{eq:I2}
I_2(t) \le  \frac{\kappa_\star}{4\varpi^\star} \int_\Omega |\grad \check\bpi(t)|^2\d\x + C.
\ee
Taking~\eqref{eq:I1}--\eqref{eq:I2} into account in~\eqref{eq:capi-interchange} provides
\be
\label{eq:inter_estimate_intern}
 \frac{\d}{\d t}\left(\int_\Omega \Pi(\check \bs^*(\x,t),\x) \d\x\right) \le 
- \frac{\kappa_\star}{2\varpi^\star} \int_\Omega |\grad \check\bpi(t)|^2\d\x + C, 
 \qquad \forall t>0.
\ee

Let us now focus on the potential (gravitational) energy. Since $\check \bs(t)$ belongs to
$\bXx\cap \bAa$ for all $t>0$, we can make use of the relation
$$
\check s_0(\x,t) = \omega(\x) - \sum_{i=1}^N \check s_i(\x,t), \qquad \text{for all }(\x,t) \in \O \times \R_+, 
$$
to write: for all $t>0$, 
$$
 \sum\limits_{i=0}^N\int_\Omega \check s_i(\x,t)\Psi_i(\x)\d\x =
  \sum\limits_{i=1}^N\int_\Omega \check s_i(\x,t)(\Psi_i-\Psi_0)(\x)\d\x + \int_\Omega \omega(\x)\Psi_0(\x)\d\x.
$$ 
This leads to 
\be\label{eq:J12}
 \frac{\d}{\d t}\left(\sum\limits_{i=0}^N\int_\Omega \check s_i(t)\Psi_i\d\x\right)=\sum\limits_{i=1}^N\int_\Omega(\Psi_i(\x)-\Psi_0(\x))\p_t s_i(\x,t)\d\x = J_1(t) + J_2(t),
\ee
where, using the equations~\eqref{eq:auxiliary_flow_interchange}, we have set 
\begin{align*}
J_1(t) = & -\sum_{i=1}^N\int_\Omega \grad (\Psi_i-\Psi_0)\cdot \bbK\grad \check s_i(t) \d\x, 
\\
J_2(t) = & \sum_{i=1}^N\int_\Omega \frac{\check s_i(t)}{\omega}\grad (\Psi_i-\Psi_0)\cdot
\bbK\grad \omega \d\x.
\end{align*}
The term $J_1$ can be estimated using~\eqref{eq:Young-delta}. More precisely, for all $\delta >0$, 
we have 
\be\label{eq:J1_1}
 J_1(t) \le  \kappa^\star \left( \delta \|\grad \check \bs^*(t)\|^2_{L^2} 
 	+ \frac{1}{\delta}\sum\limits_{i=1}^N\|\nabla(\Psi_i-\Psi_0)\|^2_{L^2} \right).
\ee
Using~\eqref{eq:nabla_s_nabla_pi} together with~\eqref{eq:H_pi_1}--\eqref{eq:H_pi_2}, 
we get that 
$$
\left\|\grad \check \bs^*\right\|^2_{L^2} 
\le \left( \frac1{\varpi_\star} \|\grad \check \bpi \|_{L^2} + |\O| M_\bphi \right)^2 
\le \frac2{(\varpi_\star)^2} \|\grad \check \bpi \|_{L^2}^2 + 2 \left(|\O| M_\bphi\right)^2.
$$
Therefore, choosing $\delta = \frac{(\varpi_\star)^2 \kappa_\star}{8 \kappa^\star \varpi^\star}$ in~\eqref{eq:J1_1}, we infer from the regularity of $\bPsi$ that 
\be\label{eq:J1}
 J_1(t) \le  \frac{\kappa_\star}{4 \varpi^\star} \int_\Omega |\grad \check\bpi(t)|^2\d\x +C, \qquad \forall t>0.
\ee
Finally, it follows from the fact that  $\sum_{i=1}^N \check s_i \le \omega$, from the
Cauchy-Schwarz inequality, and from
{the regularity of $\bPsi,\omega$}
that 
\be\label{eq:J2}
 J_2(t) \ge - \kappa^\star \sum_{i=1}^N\| \grad \Psi_i - \grad \Psi_0\|_{L^2} \|\grad \omega\|_{L^2} = C.
\ee
Combining~\eqref{eq:J12}, \eqref{eq:J1}, and~\eqref{eq:J2} with~\eqref{eq:inter_estimate_intern}, 
we get that 
\be\label{eq:interchange_Ee}
\frac{\d}{\d t}\Ee(\check \bs(t)) \le - \frac{\kappa_\star}{4\varpi^\star} \int_\Omega 
|\grad \check\bpi(t)|^2\d\x + C, \qquad \forall t >0. 
\ee
Denote by
\be\label{eq:Fntau}
\mathcal F^{n}_\tau(\bs):=\frac{1}{2\tau}\bW^2(\bs,\bs^{n-1})+\mathcal E(\bs)
\ee
the functional to be minimized in \eqref{eq:JKO}, then gathering 
\eqref{eq:inter_estimate_W} and~\eqref{eq:interchange_Ee} provides
\begin{multline*}
\frac{\d}{\d t}\Ff_\tau^{n}(\check \bs(t))+ \frac{\kappa_\star}{4\varpi^\star}\|\grad \check \bpi\|^2_{L^2} \\
\leq C\left(1+\frac{\bW^2(\check \bs(t),\bs^{n-1})}{\tau} + \sum\limits_{i=0}^N \frac{\mathcal H_\omega(s_i^{n-1})-\mathcal H_\omega(\check s_i(t))}{\tau}\right) 
\qquad \forall t>0.
\end{multline*}
Since $\check \bs(0)=\bs^n$ is a minimizer of \eqref{eq:JKO} we must have
$$
0\leq \limsup\limits_{t \to 0^+}\left(\frac{\d}{\d t}\Ff^n_\tau(\check \bs(t))\right),
$$
otherwise $\check \bs(t)$ would be a strictly better competitor than $\bs^n$ for small $t>0$.
As a consequence, we get 
$$
\liminf\limits_{t \to 0^+}\|\grad\check \bpi(t)\|^2_{L^2}\leq 
C\limsup\limits_{t\to 0^+}\left(1+\frac{\bW^2(\check \bs(t),\bs^{n-1})}{\tau} + \sum\limits_{i=0}^N \frac{\mathcal H_\omega(s_i^{n-1})-\mathcal H_\omega(\check s_i(t))}{\tau}\right).
$$
Since $\check s_i$ belongs to $C([0,1]; L^p(\O))$ for all $p \in [1,\infty)$ (see for instance~\cite{Cont_L1}), the continuity of the Wasserstein distance and of the Boltzmann entropy with respect to strong $L^p$-convergence imply that
$$
\bW^2(\check \bs(t),\bs^{n-1}) \underset{t\to 0^+}\longrightarrow \bW^2(\bs^n,\bs^{n-1}) 
\quad \text{and}\quad
\mathcal H_\omega(\check s_i(t))  \underset{t\to 0^+}\longrightarrow\mathcal H_\omega(s_i^n).
$$
Therefore, we obtain that 
\be\label{eq:liminf_pi_eps}
\liminf\limits_{t \to 0^+}\|\grad\check \bpi(t)\|^2_{L^2}\leq C\left(
1+\frac{\bW^2(\bs^n,\bs^{n-1})}{\tau} + \sum\limits_{i=0}^N \frac{\mathcal H_\omega(s_i^{n-1})-\mathcal H_\omega(s_i^n)}{\tau}
\right).
\ee
It follows from the regularity of $\bpi$ that 
$$
\bpi(\check\bs^{*}(t), \x)=\check \bpi(t)  \underset{t\to 0^+}\longrightarrow \bpi^n = \bpi(\bs^{n*}, \x) \quad \text{in}\; 
L^p(\O).
$$
Finally, let $\left(t_\ell\right)_{\ell \ge 1}$ be a decreasing sequence tending to $0$ 
realizing the $\liminf$ in~\eqref{eq:liminf_pi_eps}, then the sequence 
$\left(\grad \check \bpi(t_\ell)\right)_{\ell\ge1}$ converges weakly in $L^2(\O)^{N\times d}$ towards 
$\grad \bpi^n$. The lower semi-continuity of the norm w.r.t. the weak convergence leads to 
\begin{multline*}
\sum\limits_{i=1}^N\|\grad \pi_i^n\|^2_{L^2} \leq \lim\limits_{\ell \to \infty}\|\grad \check \bpi(t_\ell)\|^2_{L^2} =\liminf\limits_{t\to 0^+}\|\grad \check \bpi(t)\|^2_{L^2}\\
\leq C\left(1+\frac{\bW^2(\bs^n,\bs^{n-1})}{\tau} + \sum\limits_{i=0}^N \frac{\mathcal H_\omega(s_i^{n-1})-\mathcal H_\omega(s_i^n)}{\tau}\right)
\end{multline*}
and the proof is complete.
\end{proof}

\section{The Euler-Lagrange equations and pressure bounds}\label{ssec:Euler-Lagrange}

The goal of this section is to extract informations coming from the optimality of 
$\bs^n$ in the JKO minimization~\eqref{eq:JKO}. The main difficulty consists in 
constructing the phase and capillary pressures from this optimality condition. 
Our proof is inspired from~\cite{MRS10} and makes an extensive use 
of the Kantorovich potentials. Therefore, 
we first recall their definition and some useful properties. 
We refer to \cite[\S1.2]{Santambrogio_OTAM} or \cite[Chapter 5]{Villani09} for details.

\smallskip
Let $(\nu_1, \nu_2) \in \Mm_+(\O)^2$ be two nonnegative measures with same total mass.
A pair of Kantorovich potentials $(\varphi_i ,\psi_i) \in L^1(\nu_1)\times L^1(\nu_2) $ associated to the 
measures $\nu_1$ and $\nu_2$ and to the cost function $\frac12 d_i^2$ defined by~\eqref{eq:def_d_i}, $i\in \{0,\dots, N\}$, 
is a solution of the Kantorovich {\it dual problem} 
$$
DP_i(\nu_1,\nu_2) = \max\limits_{\substack{(\varphi_i ,\psi_i) \in L^1(\nu_1)\times L^1(\nu_2) \\ \varphi_i(\x) + \psi_i(\by) \leq \frac12d_i^2(\x,\by)}} \quad \int_{\O} \varphi_i(\x) \nu_1(\x)\d\x + \int_{\O} \psi_i(\by) \nu_2(\by) \d \by. 
$$
We will use the three following important properties of the Kantorovich potentials:
\begin{enumerate}[(a)]
\item There is always duality
$$
DP_i(\nu_1,\nu_2)=  \frac12 W^2_i(\nu_1,\nu_2), \qquad \forall i \in \{0,\dots, N\}. 
$$
\item A pair of Kantorovich potentials $(\varphi_i, \psi_i)$ is $\d\nu_1 \otimes \d\nu_2$ unique, 
up to additive constants.
\item The Kantorovich potentials $\varphi_i$ and $\psi_i$ are $\frac12d^2_i$-conjugate, that is 
\begin{align*}
\varphi_i( \x ) = \inf\limits_{\by \in \O}  \frac12d^2_i(\x,\by)-\psi_i ( \by ),  &\quad \forall\,  \x \in \O,  \\ 
\psi_i(\by) = \inf\limits_{\x \in \O}  \frac12d^2_i(\x,\by)-\varphi_i(\x),  &\quad \forall\, \by \in \O.
\end{align*}
\end{enumerate}
\begin{rem}\label{rem:potlip}
Since $\O$ is bounded, the cost functions $(\x,\by)\mapsto\frac12 d^2_i(\x,\by)$, $i \in \{1,\dots, N\}$, are globally Lipschitz continuous, see \eqref{eq:d_i-euclid}.
Thus item $(c)$ shows that $\varphi_i$ and $\psi_i$ are also Lipschitz continuous.
\end{rem}
\smallskip

\subsection{A decomposition result}
\label{subsec:decomposition}

The next lemma is an adaptation of  \cite[Lemma 3.1]{MRS10} to our framework. It essentially states that, since $\bs^n$ is a minimizer of \eqref{eq:JKO}, it is also a minimizer of the linearized problem.
{
\begin{lem}
\label{lem:Euler_Lagrange_lin} 
For $n\geq 1$ and $i=0,\ldots,N$ there exist some (backward, optimal) Kantorovich potentials $\varphi_i^n$ from $s^n_i$ to $s^{n-1}_i$ such that, 
using the convention $\pi^n_0=\frac{\partial \Pi}{\partial s_0}(s_1^n,\dots,s_N^n,\x)=0$, setting
\be
\label{eq:def_Fi}
F_i^n:=\frac{\varphi^n_i}{\tau}+\pi_i^n+\Psi_i,\quad \forall  i\in\left\{0,\dots,N\right\},
\ee
and denoting
${\bf F}^n = \left(F_i^n\right)_{0\le i \le N}$, there holds
\begin{equation}
 \label{eq:min_lin_Kantorovich}
 \bs^{n}\in\underset{\bs \in \bXx \cap \bAa }{\operatorname{Argmin}} 
 \int_{\Omega}\mathbf{F}^{n}(\x)\cdot \bs (\x)\d\x.
\end{equation}
Moreover, $F_i^n\in L^\infty\cap H^1(\Omega)$ for all $i \in \{0,\dots,N\}$.
\end{lem}
\begin{proof}
 We assume first that $s^{n-1}_i(\x)>0$ everywhere in $ \O$ for all $i \in \{1,\dots, N\}$, 
 so that the Kantorovich potentials $(\varphi_i^n,\psi_i^n)$ from $s_i^n$ to $s_i^{n-1}$ are 
 uniquely determined after 
 normalizing $\varphi_i^n(\x_{\rm ref})=0$ for some arbitrary point $\x_{\rm ref}\in \Omega$ 
 (cf. \cite[Proposition 7.18]{Santambrogio_OTAM}).
 Given any $\bs = \left(s_i\right)_{1\le 0 \le N} \in \bXx\cap \bAa$ and $\eps\in(0,1)$ we define the perturbation
 $$
 \bs^\eps:=(1-\eps)\bs^n+\eps \bs.
 $$
 Note that $\bXx\cap \bAa$ is convex, thus $\bs^\eps$ is an admissible competitor for all $\eps\in(0,1)$.
 Let $(\varphi_i^\eps,\psi_i^\eps)$ be the unique Kantorovich potentials from $s_i^\eps$ to $s_i^{n-1}$, similarly normalized as $\varphi_i^\eps(\x_{\rm ref})=0$.
Then by characterization of the squared Wasserstein distance in terms of the dual Kantorovich problem we have
 $$
\begin{cases}
\displaystyle  \frac{1}{2}W_i^2(s_i^\eps,s_i^{n-1})=\int_\O \varphi_i^\eps(\x) s_i^\eps(\x) \d\x
  +\int_\O \psi_i^\eps(\boldsymbol{y})s_i^{n-1}(\boldsymbol y) \d\boldsymbol{y}, \\[8pt]
\displaystyle  \frac{1}{2}W_i^2(s_i^n,s_i^{n-1})\geq \int_\O \varphi_i^\eps(\x) s_i^n(\x) \d\x
  +\int_\O \psi_i^\eps(\boldsymbol{y})s_i^{n-1}(\boldsymbol y)\d\boldsymbol{y}.
\end{cases}
 $$
 By definition of the perturbation $\bs^\eps$ it is easy to check that 
 $s_i^\eps-s_i^n=\eps(s_i-s_i^n)$.
 Subtracting the previous inequalities we get
  \be
\frac{W_i^2(s_i^\eps,s_i^{n-1})-W_i^2(s_i^n,s_i^{n-1})}{2\tau}
 \leq 
 \frac{\eps}{\tau}\int_\O \varphi_i^\eps(s_i-s_i^n)\d\x.
 \label{eq:W2_kanto_i}
 \ee
 Denote by $\bs^{\eps*} = \left(s^\eps_1,\dots,s_N^\eps\right)$, $\bpi^\eps = \bpi(\bs^{\eps*},\cdot)$, and extend to the zero-th component $\overline{\bpi}^\eps=(0,\bpi^\eps)$. 
 The convexity of $\Pi$ as a function of $s_1,\dots,s_N$ implies that 
 \begin{multline} \label{eq:internal_kanto}
 \int_\O \left(\Pi(\bs^{n*},\x) - \Pi(\bs^{\eps*},\x)\right) \d\x 
 \ge
 \int_\O \bpi^\eps \cdot \left(\bs^{n*} - \bs^{\eps*}\right) \d\x\\
 =\int_\O \overline{\bpi}^\eps \cdot \left(\bs^{n} - \bs^{\eps}\right) \d\x
 = - \eps \int_\O \overline\bpi^\eps \cdot \left(\bs - \bs^n\right) \d\x.
 \end{multline}
For the potential energy, we obtain by linearity that
 \be
 \label{eq:ext_kanto}
 \int_\O \left(\bs^\eps - \bs^n\right)\cdot \bPsi\, \d\x = 
  \eps\int_\Omega(\bs-\bs^{n})\cdot \bPsi\d\x.
 \ee
 Summing \eqref{eq:W2_kanto_i}--\eqref{eq:ext_kanto}, dividing by $\eps$, and recalling that $\bs^n$ minimizes the functional $\mathcal F^n_\tau$ defined by \eqref{eq:Fntau}, we obtain
\be\label{eq:tralala-eps}
  0\leq \frac{\mathcal F_\tau^n (\bs^\eps)-\mathcal F^n_\tau(\bs^n)}{\eps}
  \leq \sum\limits_{i=0}^N\int_\Omega\left(\frac{\varphi^\eps_i}{\tau}+\overline\pi_i^\eps+\Psi_i\right) (s_i-s_i^n)\d\x
 \ee
 for all $\bs\in \bXx\cap\bAa$ and all $\eps\in (0,1)$.
 Because $\Omega$ is bounded, any Kantorovich potential is globally Lipschitz with bounds uniform in $\eps$ 
 (see for instance the proof of \cite[Theorem 1.17]{Santambrogio_OTAM}).
Since $\bs^\eps$ converges uniformly towards $\bs^n$ when $\eps$ tends to $0$, we infer 
from~\cite[Theorem 1.52]{Santambrogio_OTAM} that $\varphi_i^\eps$
converges uniformly towards $\varphi_i^n$ as $\eps$ tends to $0$, where 
$\varphi_i^n$ is a Kantorovich potential form $s_i^n$ to $s_i^{n-1}$.
Moreover, since $\bpi$ is uniformly continuous in $\bs$, we also know that $\bpi^\eps$ converges uniformly 
towards $\bpi^n$ and thus the extension to the zero-th component $\overline{\bpi}^\eps\to \overline\bpi^n=(0,\bpi^n)$ as well.
Then we can pass to the limit in~\eqref{eq:tralala-eps} and infer that
\be\label{eq:tralala-0}
0 \le  \int_\O {\bf F}^{n}\cdot(\bs - \bs^{n})\d\x, 
\qquad \forall \bs \in \bXx \cap \bAa
\ee
and \eqref{eq:min_lin_Kantorovich} holds.

If $s_i^{n-1}>0$ does not hold everywhere we argue by approximation.
Running the flow \eqref{eq:auxiliary_flow_interchange} for a short time $\delta>0$ starting from $\bs^{n-1}$, we construct an approximation $\bs^{n-1,\delta}=(s_0^{n-1,\delta},\dots,s_N^{n-1,\delta})$ converging to $\bs^{n-1}=(s_0^{n-1},\dots,s_N^{n-1})$
in $L^1(\Omega)$ as $\delta$ tends to $0$.
By construction $\bs^{n-1,\delta} \in \bXx\cap\bAa$, and it follows from the strong maximum principle that 
$s_i^{n-1,\delta}>0$ in $\ov\O$ for all $\delta>0$.
By Proposition~\ref{prop:exist-JKO} there exists a unique minimizer $\bs^{n,\delta}$ to the functional 
$$
\Ff_\tau^{n,\delta}: \begin{cases}
\bXx\cap \bAa \to \R_+ \\
 \bs \mapsto  \frac{1}{2\tau} \bW^2(\bs, \bs^{n-1,\delta}) + \Ee(\bs)
 \end{cases}
$$
Since $\bs^{n-1,\delta} >0$, there exist 
unique Kantorovich potentials  $(\varphi_i^{n,\delta},\psi_i^{n,\delta})$ from $s_i^{n,\delta}$ to $s_i^{n-1,\delta}$. 
This allows to construct ${\bf F}^{n,\delta}$ using~\eqref{eq:def_Fi} where $\varphi_i^n$ (resp. 
$\pi_i^n$) has been replaced by $\varphi_i^{n,\delta}$ (resp. $\pi_i^{n,\delta}$).
Thanks to the above discussion, 
\be\label{eq:tralala-delta}
0 \le \int_\O {\bf F}^{n,\delta*}\cdot(\bs^{*} - \bs^{n,\delta*})\d\x, 
\qquad \forall \bs^* \in {\bXx}^* \cap {\bAa}^*.
\ee
We can now let $\delta$ tend to $0$.
Because of the time continuity of the solutions to~\eqref{eq:auxiliary_flow_interchange}, we know that $\bs^{n-1,\delta}$ converges towards $\bs^{n-1}$ in $L^1(\O)$.
On the other hand, from the definition of $\bs^{n,\delta}$ and Lemma~\ref{lem:flow_interchange_one_step} (in particular \eqref{eq:flow_interchange_one_step} with $s^{n-1,\delta},s^{n,\delta},\bpi^{n,\delta}$ instead of $s^{n-1},s^n,\bpi^{n}$) we see that $\bpi^{n,\delta}$ is bounded in $H^1(\Omega)^{N+1}$ uniformly in $\delta>0$.
Using next the Lipschitz continuous~\eqref{eq:H_pi_2} of $\bphi$, 
one deduces that $\bs^{n,\delta}$ is uniformly bounded in $H^1(\O)^{N+1}$.
Then, thanks to Rellich's compactness theorem, we can assume 
that $\bs^{n,\delta}$ converges strongly in $L^2(\O)^{N+1}$ as $\delta$ tends to $0$.
By the strong convergence $\bs^{n-1,\delta}\to \bs^{n-1}$ and standard properties of the squared Wasserstein distance, one readily checks that
$\Ff_\tau^{n,\delta}$ $\G$-converges towards $\Ff_\tau^n$, and we can therefore identify the limit of $\bs^{n,\delta}$ 
as the unique minimizer $\bs^n$ of $\Ff_\tau^n$. 
Thanks to Lebesgue's dominated convergence theorem, we also infer that $\pi_i^{n,\delta}$ converges  in $L^2(\O)$ towards $\pi_i^n$.
Using once again the stability of the Kantorovich potentials 
\cite[Theorem 1.52]{Santambrogio_OTAM}, we know that $\varphi_i^{n,\delta}$ converges 
uniformly towards some Kantorovich potential $\varphi_i^n$. Then we can pass to the limit in
\eqref{eq:tralala-delta} and claim that~\eqref{eq:tralala-0} is satisfied even when some 
coordinates of $\bs^{n-1}$ vanish on some parts of $\O$. 

Finally, note that since the Kantorovich potentials $\varphi_i^n$ are Lipschitz continuous 
and because $\pi^n_i\in H^1$ 
(cf. Lemma~\ref{lem:flow_interchange_one_step}) and $\bPsi$ is smooth, we have $F_i^n\in H^1$.
Since the phases are bounded $0\leq s_i^n(\x)\leq \omega(\x)$ and $\bpi$ is continuous we have $\bpi^n\in L^\infty$, thus $F_i^n\in L^\infty$ as well and the proof is complete.
\end{proof}

We can now suitably decompose the vector field 
${\bf F}^n = \left(F_i^n\right)_{0\le i \le N}$ defined by \eqref{eq:def_Fi}.

\begin{coro}\label{coro:decomposition}
Let $\bF^n=(F_0^n,\dots,F_N^n)$ be as in Lemma~\ref{lem:Euler_Lagrange_lin}.
There exists $\boldsymbol{\alpha}^n \in \R^{N+1}$ such that, setting $\lambda^n(\x):=\min\limits_j(F_j^n(\x)+\alpha_j^n)$, there holds $\lambda^n\in H^1(\O)$ and
\begin{align}
\label{eq:decomp_si_+-}
&F_i^n +\alpha_i^n =  \lambda^n\;\;\d s_i^n-\text{a.e. in }\Omega,  &\forall i \in \{0,\dots,N\}, \\ 
\label{eq:nabla_lambda_i}
&\grad F_i^n=\grad \lambda^n \;\;\d s_i^n-\text{a.e. in }\Omega,  &\forall i \in \{0,\dots,N\}.
\end{align}
 \end{coro}
\begin{proof}
By Lemma~\ref{lem:Euler_Lagrange_lin} we know that $\bs^n$ minimizes $\bs\mapsto \int \bF^n\cdot \bs$ among all admissible $\bs\in \bXx\cap\bAa$.
Applying the multicomponent bathtub principle, Theorem~\ref{theo:bathtub} in appendix, we infer that there exists $\balpha^n=(\alpha_0^n,\dots,\alpha^n_N)\in\R^{N+1}$ such that $F_i^n+\alpha_i^n=\lambda^n$ for $\d s_i^n$-a.e. $\x\in\Omega$ and $\lambda^n=\min\limits_j(F_j^n+\alpha_j^n)$ as in our statement.
 Note first that $\lambda^n\in H^1(\Omega)$ as the minimum of finitely many $H^1$ functions $F_0,\dots,F_N\in H^1(\O)$.
 From the usual Serrin's chain rule we have moreover that
 $$
 \grad \lambda^n=\grad\min\limits_j(F_j^n+\alpha_j^n)=\grad F_i .\chi_{[F_i^n+\alpha_i^n=\lambda^n]},
 $$
 and since $s_i^n=0$ inside $[F_i^n+\alpha_i^n\neq \lambda^n]$ the proof is complete.
\end{proof}
}

\subsection{The discrete capillary pressure law and pressure estimates}\label{subsec:pressures}

In this section, some calculations in the Riemannian settings $(\O,d_i)$ will be carried out.
In order to make them as readable as possible, we have to introduce a few basics. 
We refer to~\cite[Chapter 14]{Villani09} for a more detail presentation.

Let $i \in \{0,\dots, N\}$, then consider the Riemannian geometry $(\O, d_i)$, and let $\x \in \O$, 
then we denote by $g_{i,\x}: \R^d \times \R^d \to \R$ the local metric tensor defined by 
\begin{equation*}
\label{eq:g_i}
g_{i,\x}(\bv, \bv) = \mu_i \bbK^{-1}(\x) \bv\cdot \bv = \bbG_i(\x) \bv\cdot \bv, \qquad \forall \bv \in \R^d.
\end{equation*}
In this framework, the gradient $\grad_{g_i} \varphi$ of a function $\varphi \in \Cc^1(\O)$ is defined by 
$$
\varphi(\x+h \bv ) = \varphi(\x)  + h g_{i,\x}\left(\grad_{g_{i,\x}} \varphi (\x), \bv\right) + o(h), \qquad 
\forall \bv \in \bbS^{d-1}, \; \forall \x \in \O.
$$
It is easy to check that this leads to the formula 
\be\label{eq:grad_gi}
\grad_{g_{i}} \varphi = \frac1{\mu_i} \bbK \grad \varphi,
\ee
where  $\grad \varphi$ stands for the usual (euclidean) gradient. The formula~\eqref{eq:grad_gi} can be 
extended to Lipschitz continuous functions $\varphi$ thanks to Rademacher's theorem. 

For $\varphi$ belonging to $\Cc^2$, we can also define the Hessian $D^2_{g_i} \varphi$ of $\varphi$ in 
the Riemannian setting by 
$$
g_{i,\x}\left(D^2_{g_i}\varphi(\x)\cdot \bv,\bv\right)=\left.\frac{d^2}{dt^2}\varphi(\boldsymbol \gamma_t)\right|_{t=0}
$$
for any geodesic $\boldsymbol \gamma_t=\exp_{i,\x}(t\bv)$ starting from $\x$ with initial speed $\bv\in T_{i,\x}\O$. 

\smallskip
Denote by $\varphi_i^n$ the backward Kantorovich potential sending $s_i^n$ to $s_i^{n-1}$
associated to the cost $\frac12d_i^2$.
By the usual definition of the Wasserstein distance through the Monge problem, one has 
$$
W_i^2(s_i^{n}, s_i^{n-1}) = \int_\O d_i^2(\x, \mathbf{t}_i^n(\x)) s_i^n(\x) \d\x, 
$$
where  $\mathbf{t}_i^n$ denotes the optimal map sending $s_i^n$ on $s_i^{n-1}$. 
It follows from \cite[Theorem 10.41]{Villani09} that 
\be\label{eq:ti}
\mathbf{t}_i^n(\x) = \mathrm{exp}_{i,\x}\left(-\grad_{g_i} \varphi_i^n(\x)\right), \qquad \forall \x \in \O.
\ee
Moreover, using the definition of the exponential and the relation~\eqref{eq:grad_gi}, one gets 
that 
$$
d_i^2(\x, \mathrm{exp}_{i,\x}\left(-\grad_{g_i} \varphi_i^n(\x)\right) = 
	g_{i,\x} \left(\grad_{g_i} \varphi_i^n(\x), \grad_{g_i} \varphi_i^n(\x)\right) = 
	 \frac{1}{\mu_i} \bbK(\x)\grad \varphi_i^n(\x) \cdot \grad \varphi_i^n(\x).
$$
This yields the formula
\be\label{eq:BB00}
W_i^2(s_i^{n}, s_i^{n-1})  = \int_\O \frac{s_i^n}{\mu_i} \bbK \grad \varphi_i^n \cdot \grad \varphi_i^n \d\x, 
\qquad \forall i \in \{0,\dots, N\}.
\ee

We have now introduced the necessary material in order to reconstruct 
the phase and capillary pressures. This is the purpose of the following 
Proposition~\ref{ajustement} and of then Corollary~\ref{cor:discrete_phase_pressure}
\begin{prop} \label{ajustement}
For $n\geq 1$ let $\varphi_i^n:s^n_i\to s_i^{n-1}$ be the (backward) Kantorovich potentials from Lemma~\ref{lem:Euler_Lagrange_lin}.
There exists $\bh = \left(h_0^n,\dots,h_N^n\right) \in 
H^1(\Omega)^{N+1}$ such that
\begin{enumerate}[(i)]
 \item 
 $\grad h_i^n=-\frac{\grad\varphi_i^n}{\tau}$ for $\d s_i^n$-a.e. $\x\in \Omega$
 \item
 $h_i^n(\x)-h_0^n(\x)=\pi_i^n(\x)+\Psi_i(\x)-\Psi_0(\x)$ for $\d \x$-a.e. $\x\in \Omega$, $i \in \{1,\dots,N\}$
 \item
 there exists $C$ depending only on $\O, \Pi, \omega, \bbK, 
{(\mu_i)}_i,$ and $\bPsi$ such that, for all $n\ge1$ and all $\tau>0$, one has 
$$
\|\bh^n\|^2_{H^1(\Omega)^{N+1}}
\leq 
C\left(1+\frac{\bW^2(\bs^n,\bs^{n-1})}{\tau^2} + \sum\limits_{i=0}^N \frac{\mathcal H_\omega(s_i^{n-1})-\mathcal H_\omega(s_i^n)}{\tau}\right).
$$
\end{enumerate}
\label{prop:pressures}
\end{prop}
\begin{proof}
{
Let $\varphi^n_i$ be the Kantorovich potentials from Lemma~\ref{lem:Euler_Lagrange_lin} and $F_i^n\in L^\infty\cap H^1(\O)$ as in \eqref{eq:def_Fi}, as well as $\boldsymbol{\alpha}^n \in \R^{N+1}$ and $\lambda^n=\min\limits_j(F_j^n+\alpha_j^n) \in L^\infty\cap H^1(\O)$ 
as in Corollary~\ref{coro:decomposition}.
Setting
 \begin{equation*}
 \label{eq:hin}
h_i^n:=-\frac{\varphi^n_i}{\tau}+F_i^n-\lambda^n, \qquad \forall i \in \{0,\dots, N\}, 
 \end{equation*}
 we have $h_i^n\in H^1(\Omega)$ as the sum of Lipschitz functions (the Kantorovich potentials $\varphi_i^n$) and $H^1$ functions $F_i^n,\lambda^n$.
 Recalling that we use the notation $\pi_0=\frac{\partial\Pi}{\partial s_0}=0$, we see from the definition \eqref{eq:def_Fi} of $F_i^n$ that
 \begin{multline}
 \label{eq:diff_h_in=pi}
h_i^n-h_0^n=  \left(F_i^n-\frac{\varphi_i^n}{\tau}\right)-\left(F_0^n-\frac{\varphi_0^n}{\tau}\right)=(\pi_i^n+\Psi_i)-(\pi_0^n+\Psi_0)=\pi_i^n+\Psi_i-\Psi_0
 \end{multline}
 for all $i\in \{1,\dots,N\}$ and $\d \x$-a.e. $x$, which is exactly our statement (ii).
 
  For (i), we simply use \eqref{eq:nabla_lambda_i} to compute
 \be\label{eq:grad-hi}
\grad h_i^n=-\frac{\grad \varphi_i^n}{\tau} + \grad(F_i^n-\lambda_i^n)=-\frac{\grad \varphi_i^n}{\tau}\quad \text{ for }\d s_i^n\text{-a.e. }\x\in \Omega, \qquad \forall i \in \{0,\dots, N\}.
 \ee
 }
In order to establish now the $H^1$ estimate (iii), let us denote
$$
\Uu_i = \left\{\x \in \O \; \middle| \; s_i^n(\x) \ge \frac{\omega_\star}{N+1} \right\}.
$$
Then since $\sum s_i^n(\x) = \omega(\x) \ge \omega_\star>0$, one gets that, 
up to a negligible set, 
\be\label{eq:cup-Uu}
\bigcup_{i=0}^N \Uu_i = \O, 
\quad \text{hence}\quad 
\left(\Uu_i\right)^c \subset \bigcup_{j\neq i} \Uu_j.
\ee
We first estimate $\grad h_0^n$.
To this end, we write
\be\label{eq:grad-h0-1}
\|\grad h^n_0\|^2_{L^2} \le \frac{1}{\kappa_\star} \int_\O \bbK \grad h_0^n \cdot \grad h_0^n \d\x  \le A + B, 
\ee
where we have set
$$
A= \frac{1}{\kappa_\star} \int_{\Uu_0} \bbK \grad h_0^n \cdot \grad h_0^n\d\x,
\qquad 
B =  \frac{1}{\kappa_\star} \int_{\left(\Uu_0\right)^c} \bbK \grad h_0^n \cdot \grad h_0^n\d\x.
$$

Owing to~\eqref{eq:grad-hi} one has $\grad h_0^n = -\frac{\grad \varphi_0}{\tau}$ on 
$\Uu_0\subset \O$, where $s_0^n\geq\frac{\omega_\star}{N+1}$. Therefore, 
$$
A \le \frac{(N+1)\mu_0}{\omega_\star \kappa_\star} \int_{\Uu_0}
\frac{s_0^n}{\mu_0} \bbK \grad h_0^n \cdot \grad h_0^n\d\x \le 
 \frac{(N+1)\mu_0}{\tau^2 \omega_\star \kappa_\star} \int_{\O}
\frac{s_0^n}{\mu_0} \bbK \grad \varphi_0^n \cdot \grad \varphi_0^n\d\x.
$$
Then it results from formula~\eqref{eq:BB00} that
\be\label{eq:pressure_A}
A \le \frac{C}{\tau^2} W_0^2(s_0^n, s_0^{n-1})
\ee
where $C$ depends neither on $n$ nor on $\tau$.
Combining~\eqref{eq:cup-Uu} and \eqref{eq:diff_h_in=pi}, we infer
$$
B \le \frac{1}{\kappa_\star} \sum_{i=1}^N \int_{\Uu_i} 
\bbK \grad [h^n_i-(\pi_i^n+\Psi_i-\Psi_0)]\cdot \grad [h^n_i-(\pi_i^n+\Psi_i-\Psi_0)] \d\x.
$$
Using $(a+b+c)^2 \le 3 (a^2 + b^2 + c^2)$ and~\eqref{eq:Kelliptic}, we get that 
\be\label{eq:pressure_B1}
B \le \frac3{\kappa_\star}  \sum_{i=1}^N 
\int_{\Uu_i} \bbK \grad h_i \cdot \grad h_i\d\x + \frac{3 \kappa^\star}{\kappa_\star} \sum_{i=1}^N
\left(\| \grad \pi_i^n \|_{L^2}^2 +\| \grad (\Psi_i-\Psi_0) \|_{L^2}^2\right).
\ee
Similar calculations to those carried out to estimate $A$ yield 
$$
\int_{\Uu_i} \bbK \grad h_i \cdot \grad h_i\d\x \le \frac{C}{\tau^2} W_i^2(s_i^n, s_i^{n-1})
$$
for some $C$ depending neither on $n, i$ nor on $\tau$.
Combining this inequality with Lemma~\ref{lem:flow_interchange_one_step} and the regularity of $\bPsi$, we get from \eqref{eq:pressure_B1} that
\be\label{eq:pressure_B}
B \le C \left( 1 + \frac{\bW^2(\bs^n, \bs^{n-1})}{\tau^2} + \sum_{i=0}^N \frac{\Hh_\omega(s_i^{n-1}) - \Hh_\omega(s_i^n)}{\tau} \right)
\ee
for some $C$ not depending on $n$ and $\tau$ (here we also used $1/\tau\leq 1/\tau^2$ for small $\tau$ in the $W^2$ terms).
Gathering~\eqref{eq:pressure_A} and 
\eqref{eq:pressure_B} in~\eqref{eq:grad-h0-1} provides 
$$
\|\grad h^n_0\|^2_{L^2} \le C\left( 1 + \frac{\bW^2(\bs^n, \bs^{n-1})}{\tau^2} + \sum_{i=0}^N \frac{\Hh_\omega(s_i^{n-1}) - \Hh_\omega(s_i^n)}{\tau} \right).
$$

Note that (i)(ii) remain invariant under subtraction of the same constant 
$h_0^n,h_i^n\leadsto h^n_0-C,h^n_i-C$, as the gradients remain unchanged 
in (i) and only the differences $h_i^n-h_0^n$ appear in (ii) for $i\in \{1\dots N\}$.
We can therefore assume without loss of generality that $\int_\O h^n_0 \d\x =0$.
Hence by the Poincar\'e-Wirtinger inequality, we get that
$$
\|h_0^n\|^2_{H^1}\leq C\|\grad h_0^n\|^2_{L^2}\leq C\left(1+\frac{\bW^2(\bs^n,\bs^{n-1})}{\tau^2} + \sum\limits_{i=0}^N \frac{\mathcal H_\omega(s_i^{n-1})-\mathcal H_\omega(s_i^n)}{\tau}\right).
$$
Finally, from (ii) $h_i^n=h_0^n+\pi_i^n+\Psi_i-\Psi_0$, the smoothness of $\bPsi$, and using again the estimate \eqref{eq:flow_interchange_one_step} for $\|\grad\bpi^n\|^2_{L^2}$ we finally get 
that for all $i \in \{1,\dots, N\}$, one has 
\begin{multline*}
\|h_i^n\|^2_{H^1}\leq C(\|h_0^n\|_{H^1}^2+\|\pi_i^n\|_{H^1}^2+\|\Psi_i\|_{H^1}^2+\|\Psi_0\|_{H^1}^2)\\
 \leq C\left(1+\frac{\bW^2(\bs^n,\bs^{n-1})}{\tau^2} + \sum\limits_{i=0}^N \frac{\mathcal H_\omega(s_i^{n-1})-\mathcal H_\omega(s_i^n)}{\tau}\right), 
\end{multline*}
and the proof of Proposition \ref{ajustement} is complete.
 \end{proof}
We can now define the phase pressures $\left(p_i^n\right)_{i=0,\dots,N}$ by setting 
\be\label{eq:phase-pressures}
p_i^n := h_i^n - \Psi_i, \qquad \forall i \in \{0, \dots, N\}.
\ee
The following corollary is a straightforward consequence of Proposition~\ref{ajustement} 
and of the regularity of $\Psi_i$. 

\begin{coro}
\label{cor:discrete_phase_pressure}
The phase pressures $\bp^n = \left(p_i^n\right)_{0 \le i \le N} \in H^1(\O)^{N+1}$ satisfy 
 \begin{equation}
\|\bp^n\|^2_{H^1(\Omega)}
\leq 
C\left(1+\frac{\bW^2(\bs^n,\bs^{n-1})}{\tau^2} + \sum\limits_{i=0}^N \frac{\mathcal H_\omega(s_i^{n-1})-\mathcal H_\omega(s_i^n)}{\tau}\right)
\label{eq:H1_estimate_pressures_p}
 \end{equation}
 for some $C$ depending only on $\O, \Pi, \omega, \bbK, 
{(\mu_i)}_i,$ and $\bPsi$ (but neither on $n$ nor on $\tau$), 
 and the capillary pressure relations are fulfilled:
 \begin{equation}
 \label{eq:discrete_capillary}
p_i^n-p_0^n=\pi_i^n, \qquad \forall i \in \{1,\dots, N\}.
 \end{equation}
\end{coro}

Our next result is a first step towards the recovery of the PDEs. 
\begin{lem}\label{lem:PDE-disc}
There exists $C$ depending depending only on $\O, \Pi, \omega, \bbK, 
{(\mu_i)}_i,$ and $\bPsi$ (but neither on $n$ nor on $\tau$) such that, for all 
$i \in \{0,\dots, N\}$ and all $\xi \in \Cc^2(\ov \O)$, one has 
\begin{multline}
\label{eq:discrete_weak_form}
\left|\int_\Omega \left(s_i^n-s^{n-1}_i\right)\xi\d\x + \tau\int_\Omega s_i^n\frac{\bbK}{\mu_i}\grad\left(p_i^n+\Psi_i\right)\cdot \grad\xi\d\x\right|\\
\le C W_i^2(s_i^n,s_i^{n-1})\|D^2_{g_i}\xi\|_{\infty}.
\end{multline}
\end{lem}
This is of course a discrete approximation to the continuity equation $\partial_t s_i=\grad\cdot( s_i\frac{\bbK}{\mu_i}\grad\left(p_i+\Psi_i\right))$.
\begin{proof}
Let $\varphi_i^{n}$ denote the (backward) optimal Kantorovich potential from Lemma~\ref{lem:Euler_Lagrange_lin} sending $s^{n}_i$ to $s^{n-1}_i$, and let $\mathbf t_i^{n}$
be the corresponding optimal map as in~\eqref{eq:ti}.
For fixed $\xi\in \mathcal{C}^2(\overline{\Omega})$ let us first Taylor expand (in the $g_i$ Riemannian framework)
$$
\left|\xi(\mathbf t^{n}_i(\x))
-\xi(\x)+\frac{1}{\mu_i}\bbK(\x) \grad\xi(\x)\cdot\grad\varphi^{n}_i(\x)\right|
\leq\frac12\|D^2_{g_i}\xi\|_\infty d^2_i(\x,\mathbf t^{n}_i(\x)).
$$
Using the definition of the pushforward $s_i^{n-1}=\mathbf t_i^{n} \# s_i^n$, we then compute
\begin{multline*}
\left|\int_\Omega (s_i^{n}(\x)-s_i^{n-1}(\x))\xi(\x)\d\x -\int_\Omega \frac{\bbK(\x)}{\mu_i}\grad\xi(x) \cdot \grad\varphi^n_i(\x) s_i^{n}(\x)\d\x \right|\\
=\left|
\int_\Omega (\xi(\x)-\xi(\mathbf t^{n}_i(\x)) s_i^{n}(\x)\d\x
 -\int_\Omega \frac{\bbK(\x)}{\mu_i}\grad\xi(x) \cdot \grad\varphi^n_i(\x) s_i^{n}(\x)\d\x\right|\\
 \leq \int_\Omega  \frac 12\|D^2_{g_i}\xi\|_\infty d_i^2(\x,\mathbf t_i^{n}(\x))s_i^{n}(\x)\d\x=\frac 12 \|D^2_{g_i}\xi\|_\infty W_i^2(s_i^n,s_i^{n-1}).
\end{multline*}
From Proposition~\ref{prop:pressures}(i) we have $\grad\varphi_i^n=-\tau\grad h_i^n$ for $\d s_i^n$ 
a.e. $\x\in\Omega$, thus by the definition~\eqref{eq:phase-pressures} of $p_i^n$, we get 
$\grad\varphi^n=-\tau\grad(p_i^n+\Psi_i)$.
Substituting in the second integral of the left-hand side gives exactly \eqref{eq:discrete_weak_form} and the proof is complete.
\end{proof}

\section{Convergence towards a weak solution}\label{sec:convergence}
The goal is now to prove the convergence of the piecewise constant inteprolated solutions $\bs^\tau$, defined by~\eqref{eq:stau}, towards 
a weak solution $\bs$ as $\tau\to 0$.
Similarly, the $\tau$ superscript denotes the piecewise constant interpolation of any previous discrete quantity (e.g. $p_i^\tau(t)$ stands for the piecewise constant time interpolation of the discrete pressures $p^n_i$).
In what follows, we will also use the notations 
$\bs^{\tau*} = (s_1^\tau, \dots, s_N^\tau) \in L^\infty((0,T);\bXx^*)$ and 
$\bpi^\tau = \bpi(\bs^{\tau*},\x)$.
\subsection{Time integrated estimates}\label{ssec:estimates-time}
We immediately deduce from \eqref{eq:discrete_12_holder} that 
\begin{equation}
\label{eq:approx_12_Holder}
\bW(\bs^\tau(t_2),\bs^\tau(t_1))\leq C|t_2- t_1+\tau|^{\frac{1}{2}}, \qquad \forall\; 0\leq t_1\leq t_2\leq T.
\end{equation}
From the total saturation $\sum\limits_{i=0}^Ns_i^n(\x)=\omega(\x) \leq \omega^\star$ and $s_i^\tau\geq 0$, we have 
the $L^\infty$ estimates
\begin{equation}
\label{eq:pointwise_si_omega}
0 \le s^\tau_i(\x,t) \le \omega^\star \quad \text{a.e. in $Q$ for all $i \in \{0,\dots, N\}$}.
\end{equation}
\begin{lem}
 \label{lem:integrated_pressure_L2H1}
 There exists $C$ depending only on $\O, T, \Pi, \omega, \bbK, {(\mu_i)}_i$, and $\bPsi$
 such that 
 \begin{equation}
 \label{eq:integrated_pressure_L2H1}
  \|\bp^\tau\|^2_{L^2\left((0,T);H^1(\O)^{N+1}\right)}+\|\bpi^\tau\|^2_{L^2\left((0,T);H^1(\O)^{N}\right)} 
  \leq C.
 \end{equation}
\end{lem}
\begin{proof}
Summing \eqref{eq:H1_estimate_pressures_p} from $n=1$ to $n=N_\tau :=\lceil T/\tau\rceil$, we get
\begin{align*}
 \|\bp^\tau\|^2_{L^2(H^1)} =&\; \sum\limits_{n=1}^{N_\tau}\tau \|\bp^n\|^2_{H^1}\\
 \leq&\; C\sum\limits_{n=1}^{N_\tau}\tau\left(1+\frac{\bW^2(\bs^n,\bs^{n-1})}{\tau^2} + \sum\limits_{i=0}^{N_\tau}\frac{\mathcal H_\omega(s_i^{n-1})-\mathcal H_\omega(s_i^n)}{\tau}\right)\\
 \leq&\; C\left((T+1)+\sum_{n=1}^{N_\tau}\frac{\bW^2(\bs^n,\bs^{n-1})}{\tau}+\sum\limits_{i=0}^N
 \left(\mathcal H_\omega(s_i^{0})-\mathcal H_\omega(s_i^{N_\tau})\right)\right).
\end{align*}
We use that 
$$
0 \ge \Hh_\omega(s) \ge - \frac{1}{e}\|\omega\|_{L^1} \ge -\frac{|\O|}e, \qquad \forall s \in L^\infty(\O)
\text{ with } 0\le s \le \omega
$$
together with the total square distance estimate~\eqref{eq:tot_square_dist} to infer that
$
\|\bp\|^2_{L^2(H^1)} \le C.
$
The proof is identical for the capillary pressure $\bpi^\tau$ (simply summing the one-step estimate from Lemma~\ref{lem:flow_interchange_one_step}).
\end{proof}

\subsection{Compactness of approximate solutions}\label{ssec:compact}
We denote by $H'=H^1(\Omega)'$.
\begin{lem}
\label{lem:time_equicont_H*}
For each $i \in \{0,\dots, N\}$, there 
exists $C$ depending only on $\O$, $\Pi$, $\bPsi$, $\bbK$, and $\mu_i$ (but not on $\tau$) 
such that
$$
\|s_i^\tau(t_2)-s_i^\tau(t_1)\|_{H'} \leq C |t_2-t_1+\tau|^{\frac 12}, \qquad 
\forall \;0\leq t_1\leq t_2\leq T.
$$
\end{lem}
\begin{proof}
Thanks to~\eqref{eq:pointwise_si_omega}, we can apply  \cite[Lemma 3.4]{MRS10} to get 
$$
\left|\int_\Omega f \{s_i^\tau(t_2)-s_i^\tau(t_1)\}\d\x\right|\leq \|\grad f\|_{L^2(\Omega)}
W_{\rm ref}(s_i^\tau(t_1),s_i^\tau(t_2)), \qquad \forall f \in H^1(\O).
$$
Thus by duality and thanks to the distance estimate \eqref{eq:approx_12_Holder} and to the 
lower bound in~\eqref{eq:W_i-equiv}, we obtain that 
$$
\|s_i^\tau(t_2)-s_i^\tau(t_1)\|_{H'}\leq W_{\rm ref}(s_i^\tau(t_1),s_i^\tau(t_2))\leq C 
W_i(s_i^\tau(t_1),s_i^\tau(t_2))\leq C|t_2-t_1+\tau|^{\frac 12}
$$
for some $C$ depending only on $\O$, $\Pi$, $\left(\rho_i\right)_i$, 
$\g$, $\left(\mu_i\right)_i$, $\bbK$.
\end{proof}
From the previous equi-continuity in time, we deduce full compactness of the capillary pressure:
\begin{lem}
\label{lem:rel_comp_capillary}
The family $(\bpi^{\tau})_{\tau> 0}$ is sequentially relatively compact in $L^2(Q)^N$. 
\end{lem}
\begin{proof}
 We use Alt \& Luckhaus' trick \cite{AL83} (an alternate solution would consist in slightly adapting the nonlinear time compactness results 
\cite{Moussa16,ACM} to our context).
Let $h>0$ be a small time shift,
then {by monotonicity} and Lipschitz continuity of the capillary pressure function $\bpi(.,\x)$
\begin{multline*}
\|\bpi^{\tau}(\cdot+h) - \bpi^{\tau}(\cdot)\|^2_{L^2((0,T-h);L^2(\O)^N)} \\
\le \frac1{\kappa_\star} \int_0^{T-h} \!\!\int_\Omega (\bpi^{\tau}(t+h,\x) - \bpi^{\tau}(t,\x))\cdot 
 (\bs^{\tau*}(t+h,\x) - \bs^{\tau*}(t,\x)) \d\x \d t \\
 \leq \frac{2\sqrt{T}}{\kappa_\star} \| \bpi^{\tau}\|_{L^2((0,T);H^1(\O)^N)} 
 \|\bs^{\tau*}(\cdot+h,\cdot) - \bs^{\tau*}\|_{L^\infty((0,T-h);H')^N}.
\end{multline*}
Then it follows from Lemmas~\ref{lem:integrated_pressure_L2H1} and~\ref{lem:time_equicont_H*} that 
there exists $C>0$, depending neither on $h$ nor on $\tau$, such that 
$$
\|\bpi^\tau(\cdot+h,\cdot) - \bpi^\tau\|_{L^2((0,T-h);L^2(\O)^N)}  \le C |h+\tau|^{1/2}.
$$
On the other hand, the (uniform w.r.t. $\tau$) $L^2((0,T);H^1(\O)^N)$- and $L^\infty(Q)^N$-estimates on $\bpi^\tau$ ensure
that 
$$
\|\bpi^\tau(\cdot,\cdot+\by)) - \bpi^\tau\|_{L^2(0,T;L^2)} \le C\sqrt{|\by|}(1+\sqrt{|\by|}), \qquad \forall\, \by \in \R^d,
$$
where $\bpi^\tau$ is extended by $0$ outside $\Omega$. 
This allows to apply Kolmogorov's compactness theorem (see, for instance,~\cite{HOH10}) and entails the desired relative compactness.
\end{proof}

\subsection{Identification of the limit}\label{ssec:identify}
In this section we prove our main Theorem~\ref{thm:main}, and the proof goes in two steps: 
we first retrieve strong convergence of the phase contents $\bs^\tau\to \bs$ and weak 
convergence of the pressures $\bp^\tau\rightharpoonup\bp$, and then use the strong-weak 
limit of products to show that the limit is a weak solution. 
All along this section, $\left(\tau_k\right)_{k\ge1}$ denotes a sequence of times steps 
tending to $0$ as $k\to\infty$.

\begin{lem}\label{lem:exists_limit}
There exist $\bs\in L^\infty(Q)^{N+1}$ 
with $\bs(\cdot,t)\in \bXx \cap \bAa$ 
for a.e. $t \in (0,T)$, and $\bp\in L^2((0,T);H^1(\O)^{N+1})$ such that, up to an 
unlabeled subsequence, the following convergence properties hold:
\begin{align}
\label{eq:ae-conv}
\bs^{\tau_k}& \underset{k\to\infty}\longrightarrow \bs \quad 
\text{
a.e. in $Q$},  \\
\label{eq:weak-conv-capi}
\bpi^{\tau_k}& \underset{k\to\infty}{-\!\!\!\rightharpoonup} 
\bpi(\bs^*,\cdot)
\quad \text{weakly in $L^2((0,T);H^1(\O)^N)$}, \\
\label{eq:weak-conv-pressures}
\bp^{\tau_k}& \underset{k\to\infty}{-\!\!\!\rightharpoonup} 
\bp \quad \text{weakly in $L^2((0,T);H^1(\O)^{N+1})$}.
\end{align}
Moreover, the capillary pressure relations~\eqref{eq:pi_i} hold.
\end{lem}

\begin{proof}
From Lemma~\ref{lem:rel_comp_capillary}, we can assume that
$\bpi^{\tau_k}\to \bz$ strongly in $L^2(Q)^N$
for some limit $\bz$, thus a.e. up to the extraction of an additional subsequence.
Since $\bz \mapsto \boldsymbol\phi(\bz,\x)=\bpi^{-1}(\bz,\x)$ is continuous, we have that
\begin{align*}
\bs^{\tau_k*} = \bphi(\bpi^{\tau_k},\x) \underset{k\to\infty}\longrightarrow 
\bphi(\bpi,\x) =: \bs^* \quad \text{a.e. in } Q.
\end{align*}
In particular, this yields 
$\bpi^{\tau_k} \underset{k\to\infty}{\longrightarrow} \bpi(\bs^*,\cdot)$ a.e. in $Q$.
Since we had the total saturation $\sum\limits_{i=0}^Ns_i^{\tau_k}(t,\x)=\omega(\x)$, 
we conclude that the first component $i=0$ converges pointwise as well. 
Therefore, \eqref{eq:ae-conv} holds. Thanks to Lebesgue's dominated convergence 
theorem, it is easy to check that $\bs(\cdot,t) \in \bXx \cap \bAa$ for a.e. $t\in (0,T)$.
The convergences \eqref{eq:weak-conv-capi} and~\eqref{eq:weak-conv-pressures} 
are straightforward consequences of Lemma~\ref{lem:integrated_pressure_L2H1}.
Lastly, it follows from~\eqref{eq:discrete_capillary} that 
$$
p^{\tau_k}_i-p^{\tau_k}_0=\bpi_i(\bs^{\tau_k*},\cdot), \qquad \forall i \in \{1,\dots, N\}, \; 
\forall k \ge 1.
$$
We can finally pass to the limit $k\to \infty$ in the above relation thanks 
to~\eqref{eq:weak-conv-capi}--\eqref{eq:weak-conv-pressures} 
and infer 
$$
p_i-p_0=\bpi_i(\bs^*,\x)\quad \text{in }L^2((0,T);H^1(\O)), 
\qquad \forall i \in \{1,\dots, N\}.
$$
which immediately implies \eqref{eq:pi_i} as claimed.
\end{proof}

\begin{lem}\label{lem:cont-W}
Up to the extraction of an additional subsequence, the limit $\bs$ of 
$\left(\bs^{\tau_k}\right)_{k\ge1}$ belongs to $\Cc([0,T]; \bAa)$ 
where $\bAa$ is equipped with the metric $\bW$. 
Moreover, 
$
\bW(\bs^{\tau_k}(t), \bs(t)) \underset{k\to \infty}\longrightarrow 0\text{ for all $t \in [0,T]$}.
$
\end{lem}
\begin{proof}
It follows from the bounds \eqref{eq:pointwise_si_omega} on $s_i$ that for all $t \in [0,T]$, the sequence 
$\left(s_i^{\tau_k}\right)_k$ is weakly compact in $L^1(\O)$. It is also compact 
in $\Aa_i$ equipped with the metric $W_{i}$ due to the 
continuity of $W_{i}$ with respect to the weak convergence in $L^1(\O)$ 
(this is for instance a consequence of~\cite[Theorem 5.10]{Santambrogio_OTAM} 
together with the equivalence of $W_i$ with $W_{\rm ref}$ stated in~\eqref{eq:W_i-equiv}).
Thanks to~\eqref{eq:approx_12_Holder}, one has 
$$
\limsup\limits_{k\to\infty} W_i\left(s_i^{\tau_k}(t_2), s_i^{\tau_k}(t_1)\right) \le |t_2 - t_1|^{1/2}, 
\qquad \forall t_1, t_2 \in [0,T].
$$
Applying a refined version of the Arzel\`a-Ascoli theorem~\cite[Prop. 3.3.1]{AGS08} 
then provides the desired result.
\end{proof}

In order to conclude the proof of Theorem~\ref{thm:main}, it only remains to 
show that $\bs=\lim \bs^{\tau_k}$ and $\bp=\lim \bp^{\tau_k}$ satisfy the weak formulation~\eqref{eq:weak}:
\begin{prop}
 \label{prop:lim_weak_sol}
Let $\left(\tau_k\right)_{k\ge1}$ be a sequence such that the convergences
in Lemmas~\ref{lem:exists_limit} and~\ref{lem:cont-W} hold.
Then the limit $\bs$ of $\left(\bs^{\tau_k}\right)_{k\ge1}$ is a weak solution in the sense 
 of Definition~\ref{Def:weak} {(with $-\rho_i\mathbf g$ replaced by $+\nabla\Psi_i$ in the general case)}.
\end{prop}
\begin{proof}
Let $0\leq t_1 \le t_2\leq T$, and denote 
$n_{j,k} = \left\lceil \frac{t_j}{\tau_k}\right\rceil$ and $\tilde t_j = n_{j,k} \tau_k$ for $j \in \{1,2\}$.
Fixing an arbitrary $\xi \in \Cc^2(\ov \O)$ and summing~\eqref{eq:discrete_weak_form} 
from $n=n_{1,k}+1$ to $n=n_{2,k}$ yields
\begin{multline}\label{eq:conv-presque}
 \int_\Omega (s_i^{\tau_k}(t_2)-s_i^{\tau_k}(t_1))\xi \d\x=\sum\limits_{n=n_{1,k}+1}^{n_{2,k}}
 	\int_\Omega (s_i^n-s_i^{n-1})\xi \d\x\\
 =- \int_{\tilde t_1}^{\tilde t_2} \int_\O \frac{s_i^{\tau_k}}{\mu_i} \bbK\grad\left(p_i^{\tau_k} + \Psi_i \right)
	  \cdot \grad\xi \d\x \d t
+\mathcal O\left(\sum\limits_{n=n_{1,k}+1}^{n_{2,k}}W_i^2(s_i^n,s_i^{n-1})\right).
\end{multline}
Since $0 \le \tilde t_j - t_j \le {\tau_k}$ and $\frac{s_i^{\tau_k}}{\mu_i} \bbK\grad\left(p_i^{\tau_k} + \Psi_i \right) \cdot \grad\xi$ is uniformly bounded in $L^2(Q)$, 
one has 
\begin{multline*}
\int_{\tilde t_1}^{\tilde t_2} \int_\O \frac{s_i^{\tau_k}}{\mu_i} \bbK\grad\left(p_i^{\tau_k} + \Psi_i \right)
	  \cdot \grad\xi \d\x \d t \\
	  =  \int_{t_1}^{t_2} \int_\O \frac{s_i^{\tau_k}}{\mu_i} \bbK\grad\left(p_i^{\tau_k} + \Psi_i \right)
	  \cdot \grad\xi \d\x \d t + \Oo(\sqrt{\tau_k}).
\end{multline*}
Combining the above estimate with the total square distance estimate~\eqref{eq:tot_square_dist} 
in~\eqref{eq:conv-presque}, we obtain
\begin{multline}\label{eq:conv-presque2}
 \int_\Omega (s_i^{\tau_k}(t_2)-s_i^{\tau_k}(t_1))\xi \d\x+ \int_{t_1}^{t_2} \int_\O \frac{s_i^{\tau_k}}{\mu_i} \bbK\grad\left(p_i^{\tau_k} + \Psi_i \right)
	  \cdot \grad\xi \d\x \d t
= \mathcal O\left(\sqrt{\tau_k}\right).
\end{multline}
Thanks to Lemma~\ref{lem:cont-W}, and since the convergence in $(\Aa_i, W_i)$ is equivalent to the 
narrow convergence of measures (i.e., the convergence in $\Cc(\ov \O)'$, see for instance 
\cite[Theorem 5.10]{Santambrogio_OTAM}), we get that 
 \be\label{eq:conv-time}
 \int_\Omega (s_i^{\tau_k}(t_2)-s_i^{\tau_k}(t_1))\xi \d\x \underset{k\to\infty}\longrightarrow 
  \int_\Omega (s_i(t_2)-s_i(t_1))\xi \d\x.
 \ee
Moreover, thanks to Lemma~\ref{lem:exists_limit}, one has 
\be\label{eq:conv-space}
\int_{t_1}^{t_2} \int_\O \frac{s_i^{\tau_k}}{\mu_i} \bbK\grad\left(p_i^{\tau_k} + \Psi_i \right)
	  \cdot \grad\xi \d\x \d t  \underset{k\to\infty}\longrightarrow 
	\int_{t_1}^{t_2} \int_\O \frac{s_i}{\mu_i} \bbK\grad\left(p_i + \Psi_i \right)
	  \cdot \grad\xi \d\x \d t.
\ee
Gathering~\eqref{eq:conv-presque2}--\eqref{eq:conv-space} yields, for all $\xi \in \Cc^2(\ov \O)$ 
and all $0\leq t_1\leq t_2\leq T$,
\be\label{eq:conv-alter}
\int_\Omega (s_i(t_2)-s_i(t_1))\xi \d\x + 
\int_{t_1}^{t_2} \int_\O \frac{s_i}{\mu_i} \bbK\grad\left(p_i + \Psi_i \right)
	  \cdot \grad\xi \d\x \d t = 0.
\ee

In order to conclude the proof, it remains to check that the formulation~\eqref{eq:conv-alter} 
is stronger the formulation~\eqref{eq:weak}.
Let $\eps>0$ be a time step 
(unrelated to that appearing in the minimization scheme~\eqref{eq:JKO}),
and set 
$L_\eps = \left\lfloor \frac{T}{\eps}\right\rfloor$.
Let $\phi \in \Cc^\infty_c(\ov \O \times [0,T))$, one sets $\phi_\ell = \phi(\cdot,\ell \eps)$ for 
$\ell \in \left\{0,\dots,L_\eps \right\}$. Since $t \mapsto \phi(\cdot, t)$ is compactly supported in $[0,T)$, 
then there exists $\eps^\star>0$ such that $\phi_{L_\eps} \equiv 0$ for all $\eps \in (0,\eps^\star]$.
Then define by 
$$\phi^\eps: 
\left\{ \begin{array}{rcl}
\ov \O \times [0,T] &\to& \R \\
(\x,t) & \mapsto & \phi_\ell(\x) \quad \text{ if } t \in [\ell \eps, (\ell+1)\eps).
\end{array}\right.
$$
Choose $t_1 = \ell \eps$, $t_2 = (\ell+1)\eps$, $\xi = \phi_\ell$ in~\eqref{eq:conv-alter} and 
sum over $\ell \in \{0,\dots, L_\eps-1\}$. This provides 
\be\label{eq:AB-conv}
A(\eps)+B(\eps) = 0, \qquad \forall \eps >0. 
\ee
where 
\begin{align*}
A (\eps)= & \sum_{\ell = 0}^{L_\eps -1} \int_\O \left(s_i((\ell+1)\eps) - s_i(\ell \eps)\right) \phi^\ell \d\x, \\
B (\eps) = &  \iint_Q\frac{s_i}{\mu_i} \bbK\grad\left(p_i + \Psi_i \right) \cdot \grad\phi^\eps \d\x \d t.
\end{align*}
Due to the regularity of $\phi$, $\grad \phi^\eps$ converges uniformly towards $\phi$ as $\eps$ tends to $0$, so that 
\be\label{eq:B-conv}
B(\eps) \underset{\eps \to 0} \longrightarrow  \iint_Q  \frac{s_i}{\mu_i} \bbK\grad\left(p_i + \Psi_i \right) \cdot \grad\phi \d\x \d t.
\ee
Reorganizing the first term and using that $\phi_{L_\eps} \equiv 0$, we get that 
$$
A(\eps) = - \sum_{\ell = 1}^{L_\eps} \eps \int_\O s_i(\ell \eps) \frac{\phi_\ell - \phi_{\ell-1}}{\eps} \d\x - \int_\O s_i^0 \phi(\cdot, 0) \d\x.
$$
It follows from the continuity of $t \mapsto s_i(\cdot,t)$ in $\Aa_i$ equipped with $W_i$ and from 
the uniform convergence of 
$$
(\x,t) \mapsto  \frac{\phi_\ell(\x) - \phi_{\ell-1}(\x)}{\eps} \quad \text{if}\; t \in [(\ell-1)\eps, \ell \eps)
$$
towards $\p_t \phi$ that 
\be\label{eq:A-conv}
A(\eps) \underset{\eps \to 0} \longrightarrow  
- \iint_Q  s_i\p_t \phi \d\x \d t - \int_\O s_i^0 \phi(\cdot, 0) \d\x.
\ee
Combining~\eqref{eq:AB-conv}--\eqref{eq:A-conv} shows that the weak formulation~\eqref{eq:weak} 
is fulfilled.
\end{proof}

\appendix

\section{A simple condition for the geodesic convexity of $(\O, d_i)$}\label{app:convex}
The goal of this appendix is to provide a simple condition on the permeability tensor in order to ensure that 
Condition~\eqref{eq:hyp_Omega_convex} is fulfilled. For the sake of simplicity, we only consider here the case 
of isotropic permeability tensors 
\be\label{eq:K-iso}
\bbK(\x) = \kappa(\x) \bbI_d, \qquad \forall \x \in \ov \O
\ee
with $\kappa_\star \le \kappa(\x) \le \kappa^\star$ for all $\x \in \ov \O$.
Let us stress that the condition we provide is not optimal.

As in the core of the paper, $\O$ denotes a convex open subset of $\R^d$ with $C^2$ boundary $\p\O$. 
For $\ov \x \in \p\O$, we denote by $\n(\ov \x)$ the outward-pointing normal.
Since $\p\O$ is smooth, then there exists $\ell_0 >0$ such that, for all $\x\in \O$ such that 
$\text{dist}(\x, \p\O) < \ell_0$, there exists a unique $\ov \x \in \p\O$ such that $\text{dist}(\x, \p\O) = |\x - \ov \x|$
(here $\text{dist}$ denotes the usual Euclidian distance between sets in $\R^d$).
As a consequence, one can rewrite 
$\x = \ov \x - \ell \n(\ov\x)$ for some $\ell \in (0,\ell_0)$.

In what follows, a function $f: \ov \O \to \R$ is said to be normally nondecreasing (resp. nonincreasing) 
on a neighborhood of $\p\O$ if there exists $\ell_1 \in (0,\ell_0]$ such that 
$\ell \mapsto f(\ov \x - \ell \n(\ov \x))$ is nonincreasing (resp. nondecreasing) on $[0, \ell_1].$

\begin{prop}\label{prop:O-conv}
Assume that:
\begin{enumerate}[(i)]
\item  the permeability field $\x \mapsto \kappa(\x)$ is normally non-increasing in a neighborhood of $\p\O$;
\item  for all $\ov \x \in \p\O$, either $\grad \kappa(\ov \x) \cdot \n(\ov \x) < 0$, or $\grad \kappa(\ov \x) \cdot \n(\ov \x)= 0$ 
and $D^2 \kappa(\ov \x)\n(\x) \cdot \n(\x) =0.$ 
\end{enumerate}
Then there exists a $C^2$ extension $\wt \kappa: \R^d \to [\frac{\kappa_\star}2, \kappa^\star]$ of $\kappa$ and a Riemannian metric 
\be\label{eq:delta-tilde}
\wt \delta(\x,\by) = \inf_{\bga\in \wt P(\x, \by)} \left( \int_0^1 \frac1{\wt \kappa(\bga(\tau))} |\bga'(\tau)|^2{\rm d}\tau\right)^{1/2}, \qquad \forall \x, \by \in \R^d
\ee
with
$
\wt P(\x,\by) = \left\{\bga \in C^1([0,1]; \R^d)\, \middle| \, \bga(0) = \x\;  \text{and}\; 
\bga(1) = \by\right\}, 
$
such that $(\O, \wt \delta_i)$ is geodesically convex. 
\end{prop}

\begin{proof}
Since $\O$ is convex, then for all $\x \in \R^d \setminus \O$, there exists a unique $\ov \x \in \p\O$ 
such that $\text{dist}(\x,\O)=|\x - \ov \x|$. Then one can extend $\kappa$ in a $C^2$ way into the whole $\R^d$ by defining 
$$
\kappa(\x) = \kappa(\ov \x) + |\x-\ov \x| \grad \kappa(\ov \x) \cdot \n(\ov \x) +\frac{|\x-\ov \x|^2}2 D^2\kappa(\ov \x) \n(\ov \x) \cdot \n(\ov \x), 
\quad \forall \x \in \R^d \setminus \O.
$$
Thanks to Assumptions (i) and (ii), the function $\ell \mapsto \kappa(\ov \x - \ell \n(\ov \x))$ is non-decreasing on $(-\infty, \ell_1]$ 
for all $\ov \x \in \p\O$. Since $\p\O$ is compact, there exists $\ell_2 >0$ such that 
$$\kappa(\ov \x - \ell \n(\ov \x) \ge \frac{\kappa_\star}2, \qquad \forall \ell \in (-\ell_2, 0].$$
Let $\rho: \R_+\to \R$ be a non-decreasing $C^2$ function such that $\rho(0) = 1$, $\rho'(0) = \rho''(0)=0$ and $\rho(\ell) = 0$ for all $\ell \ge \ell_2$. Then define 
$$
\wt \kappa(\x) = \rho(\text{dist}(\x,\O)) \kappa(\x) + (1-\rho(\text{dist}(\x,\O))) \frac{\kappa_\star}2, \qquad \forall \x \in \R^d, 
$$
so that the function $\ell \mapsto \wt \kappa(\ov \x - \ell \n (\ov \x))$ is non-increasing on $(-\infty, \ell_1)$ and bounded from below by $\frac{\kappa_\star}2$.

Let $\x, \by \in \O$, then there exists $\eps>0$ such that $\text{dist}(\x, \p\O) \ge \eps$,  $\text{dist}(\by, \p\O) \ge \eps$,  
and $\kappa$ is normally nonincreasing on $\p\O_\eps:= \{\x \in \ov \O \; | \; \text{dist}(\x, \p\O) < \eps\}.$ A sufficient condition for $(\O, \wt \delta)$ 
to be geodesic is that the geodesic $\bga_{\x,\by}^{\text{opt}}$ from $\x$ to $\by$ is such that 
\be\label{eq:geo-eps}
\text{dist}\left(\bga_{\x,\by}^{\text{opt}}(t), \p\O\right) \ge \eps, \quad \forall t \in [0,1].
\ee
In order to ease the reading, we denote by $\bga = \bga_{\x,\by}^{\text{opt}}$ {any geodesic} such that  
\be\label{eq:bga_opt-delta}
\wt \delta^2(\x,\by) =\int_0^1 \frac1{\wt \kappa(\bga(\tau))} |\bga'(\tau)|^2{\rm d}\tau. 
\ee
We define the continuous and piecewise $C^1$
path $\bga_\eps$ from $\x$ to $\by$ by setting 
\be\label{eq:bga_eps}
\bga_\eps(t) = \text{proj}_{\ov \O_\eps}
(\bga(t)), \qquad \forall t \in [0,1],
\ee
where $\ov \O_\eps := \{ \x \in \O\; | \; \text{dist}(\x,\p\O) \ge \eps\}$ is convex, 
and the orthogonal (w.r.t. the euclidian distance $\text{dist}$) projection $\text{proj}_{\ov \O_\eps}$ onto $\ov \O_\eps$ 
is therefore uniquely defined. 

Assume that Condition~\eqref{eq:geo-eps} is violated.
Then by continuity there exists a non-empty interval $[a,b] \subset [0,1]$ such that 
$$\text{dist}(\bga(t), \p\O) < \eps, \qquad \forall t \in (a,b), $$
 the geodesic between $\gamma(a)$ and $\gamma(b)$ coincides with the part of the geodesics between $\x$ and $\by$. 
Then, changing $\x$ into $\bga(a)$ and $\by$ into $\bga(b)$, we can assume without loss of generality that 
$$\text{dist}(\bga(t), \p\O) < \eps, \qquad \forall t \in (0,1).$$
It is easy to verify that 
\be\label{eq:comp-bga_eps-bga}
|\bga'_\eps(t)| \le |\bga'(t)|, \quad \forall t \in [0,1], \qquad \text{and} \quad 
|\bga'_\eps(t)| < |\bga'(t)| \;\text{on}\; (a,b)
\ee
for some non-empty interval $(a,b) \subset [0,1]$.
It follows from~\eqref{eq:delta-tilde} that 
$$
\wt \delta^2(\x,\by) \le \int_0^1  \frac1{\wt \kappa(\bga_\eps(\tau))} |\bga_\eps'(\tau)|^2{\rm d}\tau.
$$
Since $\kappa$ is normally non-increasing, one has 
$$
\wt \delta^2(\x,\by) \le \int_0^1  \frac1{\wt \kappa(\bga(\tau))} |\bga_\eps'(\tau)|^2{\rm d}\tau.
$$
Thanks to~\eqref{eq:comp-bga_eps-bga}, one obtains that 
$$
\wt \delta^2(\x,\by) <  \int_0^1  \frac1{\wt \kappa(\bga(\tau))} |\bga'(\tau)|^2{\rm d}\tau, 
$$
providing a contradiction with the optimality \eqref{eq:bga_opt-delta} of $\bga$. Thus Condition~\eqref{eq:geo-eps}
holds, hence $(\O,\delta)$ is a geodesic space. 
\end{proof} 

\section{A multicomponent bathtub principle}\label{app:bathtub}

The following theorem can be seen as a generalization of the classical scalar bathtub principle (see for instance \cite[Theorem 1.14]{LL01}). 
In what follows, $N$ is a positive integer and $\O$ denotes an arbitrary measurable subset of $\R^d$. 

\begin{thm}
\label{theo:bathtub}
Let $\omega\in L^1_+(\Omega)$, and let $\boldsymbol{m}=(m_0,\dots,m_N)\in (\R_+^*)^{N+1}$ 
 be such that $\sum_{i=0}^N m_i= \int_\O \omega \,\d\x$. We denote by 
 $$
 \bXx\cap\bAa=\left\{ \bs=(s_0,\dots,s_N) \in L^1_+(\O)^{N+1}\middle| 
 \; \int_\O s_i \d\x =m_i\text{ and }\sum_{i=0}^N s_i=\omega\,  \text{a.e. in $\O$}\right\}.
 $$
Then for any $\bF=(F_0,\dots,F_N)\in (L^\infty(\O))^{N+1}$, the functional
 $$
 \mathcal F: \bs \mapsto \int_\O \bF\cdot \bs \,\d\x
 $$
 has a minimizer in $\bXx\cap\bAa$.
 Moreover, there exists $\balpha=(\alpha_0,\dots,\alpha_N)\in\R^{N+1}$ such that, denoting
 $$
 \lambda(\x):=\min\limits_{0 \le j \le N}\{F_j(\x)+\alpha_i\}, \qquad \x \in \O,
 $$
 any minimizer $\un \bs = \left(\un{s}_0, \dots, \un{s}_N\right)$ satisfies
 $$
 F_i+\alpha_i=\lambda \qquad \d \un s_i\text{-a.e. in }\Omega,  \quad \forall i \in \{0, \dots, N\}.
 $$
\end{thm}
One can think of this as: $\un s_i=0$ in $\{F_i+\alpha_i>\lambda\}$ and $F_i+\alpha_i\geq \lambda$ everywhere, i.e., $\un s_i>0$ can only occur in the ``contact set'' $\left\{\x\middle| \;
F_i(\x)+\alpha_i=\min\limits_j (F_j(\x)+\alpha_j)\right\}$.
\begin{proof}
For the existence part, note that $\mathcal F$ is continuous for the weak $L^1$ convergence, and that $\bXx\cap\bAa$ is weakly closed.
Since $\sum s_i=\omega$ and $s_i\geq 0$ we have in particular $0\leq s_i\leq \omega\in L^1$ for all $i$ and $\bs\in \bXx\cap\bAa$.
This implies that $\bXx\cap\bAa$ is uniformly integrable, and since the mass $\|s_i\|_{L^1}=\int s_i=m_i$ is prescribed, the Dunford-Pettis theorem shows that $\bXx\cap\bAa$ is $L^1$-weakly relatively compact.
Hence from any minimizing sequence we can extract a weakly-$L^1$ converging subsequence, and by weak $L^1$ continuity the weak limit is a minimizer.

Let us now introduce a dual problem: for fixed $\balpha=(\alpha_0,\dots,\alpha_N)\in\R^{N+1}$ we denote
\begin{equation}
\label{eq:def_lambda_alpha}
\lambda_{\balpha}(\x):=\min\limits_i\{F_i(\x)+\alpha_i\}
\end{equation}
and define
$$
J(\balpha):=\int_\O \lambda_{\balpha}(\x)\omega(\x)\d\x -\sum_{i=0}^N\alpha_im_i.
$$
We shall prove below that 
\begin{enumerate}[(i)]
\item  $\sup\limits_{\balpha \in\R^{N+1}} J(\balpha)=\max\limits_{\balpha \in\R^{N+1}} J(\balpha)$ is achieved,
\item $\min\limits_{\bs\in \bXx\cap\bAa}\mathcal F (\bs) =\max\limits_{\balpha  \in\R^{N+1} }J(\balpha)$.
\end{enumerate}
The desired decomposition will then follow from equality conditions in (ii), and $\lambda(\x)=\lambda_{\ov \balpha}(\x)$ will be retrieved from any maximizer $\ov\balpha\in \operatorname{Argmax}J$.
\begin{rem}
The above dual problem can be guessed by introducing suitable Lagrange multipliers $\lambda(\x),\balpha$ for the total saturation and mass constraints, respectively, and writing the convex indicator of the constraints as a supremum over these multipliers.
Formally exchanging $\inf\sup =\sup \inf$ and computing the optimality conditions in the right-most infimum relates $\lambda$ to $\balpha$ as in \eqref{eq:def_lambda_alpha}, which in turn yields exactly the duality $\inf\limits_{\bs}\mathcal F=\max\limits_{\balpha}J$. See also Remark \ref{bath}
\end{rem}

\medskip
Let us first establish property (i).
 For all $\balpha\in\R^{N+1}$ and all $\bs\in \bXx\cap\bAa$, we first observe that
 \begin{align*}
  J(\balpha)=&\int_\O  \min\limits_j\{F_j(\x)+\alpha_j\}\omega(\x) \d\x-\sum_{i=0}^N \alpha_i m_i\\
  \nonumber
  =&\int_\O  \min\limits_j\{F_j(\x)+\alpha_j\}\sum_{i=0}^N s_i(\x) \d\x-\sum_{i=0}^N \alpha_i \int_\O s_i(\x) \d\x\\
  \nonumber
  =& \sum_{i=0}^N\int_\O \left(  \min\limits_j\{F_j(\x)+\alpha_j\} - \alpha_i \right) s_i(\x)  \d\x 
  \leq
  \int_\O \bF\cdot \bs\, \d\x =\mathcal F(\bs).
  \end{align*}
 In particular $J$ is bounded from above and 
 \begin{equation}\label{evident}
\sup\limits_{\balpha \in\R^{N+1}} J(\balpha) \leq \min\limits_{\bs\in \bXx\cap\bAa}\mathcal F (\bs).
 \end{equation}
Since $\int \omega \d\x=\sum m_i$, the function $J$ is invariant under diagonal shifts, i.e., $J(\balpha+c\mathbf{1} )=J(\balpha)$ 
for any constant $c\in \R$.
 As a consequence we can choose a maximizing sequence $\{\balpha^k\}_{k\geq 1}$ such that $\min\limits_j \alpha^k_j=0$ for all $k\geq 0$.
 Let $j(k)$ be an index such that $\alpha_{j(k)}^k=\min\limits_j \alpha^k_j=0$.
 Then, since $\balpha^k$ is maximizing {and $\omega(\x)\geq 0$}, we get, for $k$ large enough,
 \begin{align*}
  \sup J-1\leq& \;  J(\balpha^k)=\int_\O \min\limits_j\{F_j(\x)+\alpha_j^k\}\omega(\x)\d\x-\sum \alpha_i^k m_i\\
  \leq& \int_\O \big(F_{j(k)}(\x)+\underbrace{\alpha^k_{j(k)}}_{=0}\big)\omega(\x)\d\x-\sum \alpha_i^k m_i
  \leq \|\bF\|_{L^\infty}\|\omega\|_{L^1}-\sum \alpha_i^k m_i.
 \end{align*}
Thus $\sum \alpha^k_i m_i\leq C$, and since $\alpha^k_i\geq 0$ and $m_i>0$ we {deduce} that $\left(\balpha^k\right)_k$ is bounded.
Hence, up to extraction of an unlabelled subsequence, we can assume that $\balpha^k$ converges towards some $\ov \balpha \in \R^{N+1}_+$. 
The map $J$ is continuous, hence $\ov \balpha$ is a maximizer.

\medskip
Let us now focus on property (ii). 
Note from \eqref{evident} and (i) it suffices to prove the reverse inequality
$$
\max\limits_{\balpha \in\R^{N+1}} J(\balpha)\geq \min \limits_{\bs\in\bXx\cap\bAa}\mathcal F (\bs).
$$
We show below that, for any maximizer $\ov \balpha$ of $J$, we can always construct a suitable $\bs\in \bXx\cap\bAa$ such that $\mathcal F(\bs)=J(\ov \balpha)$.
This {will immediately imply the reverse inequality} and thus our claim (ii).
In order to do so, we first observe that $J$ is concave, thus the optimality condition at $\ov \balpha$ can be written in terms of superdifferentials as $\mathbf{0}_{\R^{N+1}}\in \partial J(\ov \balpha)$.
Denoting by 
$$\Lambda(\balpha)=\int_\O 
\lambda_{\balpha} \omega \d\x=\int_\O \min\limits_j\{F_j(x)+\alpha_j\}\omega(\x) \d\x$$ 
the first contribution in $J$, this optimality can be recast as
\begin{equation}
\label{eq:m_in_superdiff}
\boldsymbol{m}\in \partial \Lambda(\ov \balpha).
\end{equation}
For fixed $\x \in \O$ and by usual properties of the $\min$ function, the superdifferential $\p \lambda_{\balpha}(\x)$ of {the concave map} $\balpha \mapsto \lambda_{\balpha}(\x)$ at  $\balpha \in \R^{N+1}$ is characterized by 
$$
\p \lambda_{\balpha}(\x) = \left\{ 
\boldsymbol \theta \in \R_+^{N+1} \; \middle| \;   \sum_{i= 0}^N \theta_i = 1, \;\text{and}\;
\theta_i = 0 \; \text{if}\; F_i(\x) + \alpha_i > \lambda_{\balpha}(\x)
\right\}.
$$
Therefore, it follows from the extension of the formula of differentiation under the integral to the non-smooth case (cf.  \cite[Theorem 2.7.2]{clarke}) that 
\be\label{eq:clarke}
\partial \Lambda(\balpha)= \left\{
\boldsymbol w \in \R_+^{N+1} \;\middle| \;
\boldsymbol w = \int_\O \boldsymbol \theta(\x) \omega(\x) \d\x \; \text{with}\; \boldsymbol \theta(\x) \in \p \lambda_{\balpha}(\x)\;\text{a.e. in}\; \O 
\right\}.
\ee
The optimality criterion \eqref{eq:m_in_superdiff} at {any maximizer }$\ov \balpha$ gives the existence 
of some function $\boldsymbol{\theta}$ as in \eqref{eq:clarke} such that
$$
m_i=\int_\O \theta_i(\x)\omega(\x) \d\x, \qquad \forall i \in \{0, \dots, N\}.
$$
Defining 
\be\label{eq:s-theta}
s_i(\x):=\theta_i(\x)\omega(\x), \qquad \forall i \in \{0,\dots, N\},
\ee
 we have by construction that $s_i\geq 0$, $\int s_i=m_i$, and $\sum s_i=(\sum_i\theta_i)\omega=\omega$ a.e, thus $\bs \in \bXx\cap\bAa$.
Exploiting again $\sum s_i=\omega$ as well as the crucial property that  $\theta_i=0$ a.e. in $\{\x \; | \; F_i+\ov \alpha_i>\lambda_{\ov \balpha}\}$, or in other words that $F_i+\ov \alpha_i=\lambda_{\ov \balpha}$ for $\d s_i$-a.e $\x\in\O$, we get
\begin{multline*}
J(\ov \balpha)=\int_\O \lambda_{\ov \balpha}\omega \d\x - \sum_{i=0}^N \ov \alpha_im_i
=\sum_{i=0}^N \int_\O \lambda_{\ov \balpha} s_i \d\x - \sum_{i=0}^N \ov \alpha_im_i\\
=\sum_{i=0}^N \int_\O (F_i+\ov \alpha_i) s_i \d\x - \sum_{i=0}^N \ov \alpha_im_i =\mathcal F(\bs)
\end{multline*}
as claimed. Therefore $\bs$ constructed by~\eqref{eq:s-theta} is a minimizer of $\Ff$ and 
\be\label{eq:min-max}
J(\ov \balpha) = \Ff( \un \bs).
\ee

\medskip
In order to finally retrieve the desired decomposition, {choose any minimizer} $\un \bs \in \bXx\cap \bAa$ of $\Ff$ and 
{any} maximizer $\ov \balpha \in \R^{N+1}$ of $J$.
Then it follows from~\eqref{eq:min-max} that 
$$
0 =  \Ff(\un \bs) - J(\ov \balpha) = \sum_{i=0}^N \int_\O F_i \un s_i \d\x - \int_\O \lambda_{\ov \balpha} \omega \d\x 
+ \sum_{i=0}^N \ov \alpha_i m_i. 
$$
Using once again that $\int \un s_i = m_i$ and $\sum_i \un s_i = \omega$, we get that 
$$
\sum_{i=0}^N \int_\O \left(F_i + \ov \alpha_i - \lambda_{\ov \balpha}\right) \un s_i \d\x = 0. 
$$
By definition of $ \lambda_{\ov \balpha} $ the above integrand is nonnegative, hence $F_i + \ov \alpha_i = \lambda_{\ov \balpha}$ a.e. in $\{ \un s_i >0 \}$.
\end{proof}

\begin{rem}\label{bath}
To understand the dual problem one chan think the function $F_i$ as $N+1$ bathub that can be translated vertically. 
The translation of each bathtub is given by $\balpha_i$. 
Once these translations are given one just wants to fill the bathubs starting from the bottom (that is $\lambda_{\balpha}$),
while satisfying the global saturation and mass constraints. 
For an optimal translation vector $\balpha$, each phase $i$ contributes at $\x$ with a ratio $\theta_i(\x)$ as in \eqref{eq:s-theta}. 
\end{rem} 

\subsection*{Acknowledgements}
This project was supported by the ANR GEOPOR project (ANR-13-JS01-0007-01). 
CC also aknowledges the support of Labex CEMPI (ANR-11-LABX-0007-01).
TOG was supported by the ANR ISOTACE project (ANR-12-MONU-0013).
LM was supported by the Portuguese Science Fundation through FCT fellowship
SFRH/BPD/88207/2012 and the UT Austin | Portugal CoLab project.
Part of this work was carried out during the stay of CC and TOG at 
CAMGSD, Instituto Superior T\'ecnico, Universidade de Lisboa. 
The authors whish to thank Quentin M\'erigot for fruitful discussion.

\def\cprime{$'$}

\end{document}